\documentclass[final,reqno]{elsarticle}
\usepackage{mathrsfs}
\usepackage{amssymb,amsmath,amsfonts,latexsym}
\usepackage{amsthm}
\usepackage{graphicx,float,epsfig,color,fancyhdr}
\usepackage{subcaption}
\usepackage{framed}
\usepackage{tabularx,colortbl}
\definecolor{LightCyan}{rgb}{0.88,1,1}
\definecolor{Gray}{gray}{0.85}

\newtheorem{lemma}{Lemma}[section]
\newtheorem{remark}{Remark}[section]
\newtheorem{proposition}{Proposition}[section]
\newtheorem{theorem}{Theorem}[section]

\newtheorem{problem}{Problem}

\def\CT{\mathcal{T}}
\def\E{K}

\def\HuE{H^1(\E)}

\def\HutO{H^{1+t}(\O)}
\def\Hcd{H_0^2(\O)}
\def\Hcu{H_0^1(\O)}

\def\LO{L^2(\O)}
\def\N{\mathbb{N}}
\def\O{\Omega}
\def\P{\mathbb{P}}
\def\R{\mathbb{R}}

\def\hdel{\widehat{\delta}}

\def\l{\lambda}

\def\sp{\mathop{\mathrm{sp}}\nolimits}

\def\Vh{V_h}
\def\VK{V^{\E}_h}

\def\HdoO{{H_0^2(\O)}}

\def\HdsO{{H^{2+s}(\O)}}
\def\HdoK{{H^{2}(\E)}}
\def\WK{W^{\E}_h}
\def\Wh{W_h}

\def\sumkth{\sum_{K\in \mathcal{T}_h}}

\journal{}
\date{\today}

\begin{document}
\begin{frontmatter}

\title{A Virtual Element Method for the Transmission Eigenvalue Problem}

\author[1,2]{David Mora}
\ead{dmora@ubiobio.cl}
\address[1]{GIMNAP, Departamento de Matem\'atica,
Universidad del B\'io-B\'io, Casilla 5-C, Concepci\'on, Chile.}
\address[2]{Centro de Investigaci\'on en Ingenier\'ia Matem\'atica
(CI$^2$MA), Universidad de Concepci\'on, Concepci\'on, Chile.}
\author[2,3]{Iv\'an Vel\'asquez}
\ead{ivelasquez@ing-mat.udec.cl}
\address[3]{Departamento de Ingenier\'ia Matem\'atica,
Universidad de Concepci\'on, Concepci\'on, Chile.}

\begin{abstract} 
In this paper, we analyze a virtual element method (VEM)
for solving a non-selfadjoint fourth-order eigenvalue
problem derived from the transmission eigenvalue problem.
We write a variational formulation and propose
a $C^1$-conforming discretization by means of the VEM.
We use the classical approximation
theory for compact non-selfadjoint operators to obtain
optimal order error estimates for the eigenfunctions
and a double order for the eigenvalues.
Finally, we present some numerical experiments
illustrating the behavior of the virtual scheme
on different families of meshes. 
\end{abstract}


\end{frontmatter}


\setcounter{equation}{0}
\section{Introduction}
\label{SEC:INTR}

In this work, we study a Virtual Element Method for an
eigenvalue problem arising in scattering theory.
The {\it Virtual Element Method} (VEM), introduced in
\cite{BBCMMR2013,BBMR2014}, is a generalization
of the Finite Element Method which is characterized
by the capability of dealing with very general
polygonal/polyhedral meshes, and it also permits
to easily implement highly regular discrete spaces.
Indeed, by avoiding the explicit construction of
the local basis functions, the VEM can easily handle
general polygons/polyhedrons without complex
integrations on the element
(see \cite{BBMR2014} for details on the coding aspects of the method).
The VEM has been developed and analyzed for many problems, see for instance
\cite{ABMV2014,ABSV2016,BBM,BLM2015,BLV-M2AN,BMR2016,BBBPS2016-0,BGS17,BM12,CG2017,CGS17,CMS2016,ChM-camwa,FS-M2AN18,PPR15,V-m3as18}.
Regarding  VEM for spectral problems, we mention \cite{BMRR,GMV-Arxiv2018,GV-IMA2017,MR2017,MRR2015,MRV}.
We note that there are other methods that can
make use of arbitrarily shaped polygonal/polyhedral meshes,
we cite as a minimal sample of them \cite{BLMbook2014,CGH14,DPECMAME2015,TPPM10}.

Due to their important role in many application areas, there has been a
growing interest in recent years towards
developing numerical schemes for spectral problems (see \cite{Boffi}).
In particular, we are going to analyze a virtual element
approximation of the transmission eigenvalue problem.
The motivation for considering this problem is that it 
plays an important role in inverse scattering theory
\cite{CakColMonkSun2010,ColtonKress2013}.
This is due to the fact that transmission eigenvalues
can be determined from the far-field data of the scattered
wave and used to obtain estimates for the material properties
of the scattering object \cite{CakCayCol2008,CakGinHad2010}.

In recent years various numerical methods
have been proposed to solve this eigenvalue problem;
see for example the following
references~\cite{CakMonkSun2014,CRV2018,ChenGZZ2016,ColtonMonkSun2010,GengJiSunXu2016,HYBi2017NonConf,HYBi2017,Sun2011}.
In particular, the transmission eigenvalue problem
is often solved by reformulating it as a fourth-order eigenvalue
problem. In \cite{CakMonkSun2014}, a $C^1$ finite element method using Argyris
elements has been proposed, a complete analysis of the method
including error estimates are proved
using the theory for compact non-self-adjoint operators.
However, the construction of conforming finite elements for $H^2(\O)$
is difficult in general, since they usually involve a large number of degrees
of freedom (see \cite{ciarlet}).
More recently, in \cite{GengJiSunXu2016} a discontinuous
Galerkin method has been proposed and analyzed to solve the
fourth-order transmission eigenvalue problem; moreover,
in \cite{ChenGZZ2016} a $C^0$ linear finite element method
has been introduced to solve the spectral problem.

The purpose of the present paper is to introduce and analyze a $C^1$-VEM
for solving a fourth-order spectral problem derived from the
transmission eigenvalue problem. We consider a variational
formulation of the problem written in $H^{2}(\O)\times H^1(\O)$
as in \cite{CakMonkSun2014,GengJiSunXu2016}, where an auxiliary
variable is introduced to transform the problem into a linear
eigenvalue problem. Here, we exploit the capability of VEM to build highly regular
discrete spaces (see \cite{BM13,BM12}) and propose a conforming $H^{2}(\O)\times H^1(\O)$
discrete formulation, which makes use of a very simple set of degrees of freedom,
namely 4 degrees of freedom per vertex of the mesh.
Then, we use the classical spectral theory for non-selfadjoint compact operators
(see \cite{BO,Osborn1975}) to deal with the continuous and
discrete solution operators, which appear as the solution of the
continuous and discrete source problems, and whose spectra are
related with the solutions of the transmission eigenvalue problem.
Under rather mild assumptions on the polygonal meshes (made by possibly non-convex elements),
we establish that the resulting VEM scheme provides a correct approximation of the
spectrum and prove optimal-order error estimates for the eigenfunctions
and a double order for the eigenvalues.
Finally, we note that, differently from the FEM
where building globally conforming $H^2(\O)$
approximation is complicated, here the virtual space
can be built with a rather simple construction
due to the flexibility of the VEM.
In a summary, the advantages of the present virtual element discretization
are the possibility to use general polygonal meshes
and to build conforming $H^2(\O)$ approximations.

The remainder of this paper is structured as follows: In 
Section~\ref{SEC:SpProCont}, we introduce the variational formulation of the
transmission eigenvalue problem, define a solution operator and establish
its spectral characterization. In Section~\ref{SEC:DISCR_PROB},
we introduce the virtual element discrete formulation, describe
the spectrum of a discrete solution operator and establish some
auxiliary results. In Section~\ref{SEC:approximation}, we prove
that the numerical scheme provides a correct spectral approximation
and establish optimal order error estimates for the eigenvalues
and eigenfunctions using the standard theory for compact
and non-selfadjoint operators. Finally, we report some numerical
tests that confirm the theoretical analysis
developed in Section~\ref{SEC:NumRes}.

In this article, we will employ standard notations for Sobolev
spaces, norms and seminorms. In addition, we will denote by $C$ a
generic constant independent of the mesh parameter $h$, which may
take different values in different occurrences.

\setcounter{equation}{0}
\section{The transmission eigenvalue problem}\label{SEC:SpProCont}

Let $\O\subset\R^2$ be a bounded domain with polygonal boundary
$\partial \O$. We denote by $\nu$ the outward unit
normal vector to $\partial \O$ and by $\partial_\nu$ the normal derivative.
Let $n$ be a real value function in $L^{\infty}(\O)$
such that $n-1$ is strictly positive (or strictly negative)
almost everywhere in $\O$. The transmission eigenvalue
problem reads as follows:

Find the so-called transmission eigenvalue $k\in \mathbb{C}$
and a non-trivial pair of functions $(w_1,w_2)\in L^2(\O)\times L^2(\O)$,
such that $(w_1-w_2)\in H^2(\O)$ satisfying
\begin{align}
\Delta w_1+k^2n(x)w_1=0&\quad \mbox{in }\O,\label{eq1}\\
\Delta w_2+k^2w_2=0&\quad \mbox{in }\O,\label{eq2}\\
w_1=w_2&\quad \mbox{on }\partial \O,\label{eq3}\\
\partial_\nu w_1=\partial_\nu w_2&\quad \mbox{on }\partial \O.\label{eq4}
\end{align}
Now, we rewrite problem above in the following
equivalent form for $u:=(w_1-w_2)\in H_0^2(\O)$ (see \cite{CakMonkSun2014}):

Find $(k,u)\in\mathbb{C}\times H_0^2(\O)$ such that
\begin{align}\label{eqfourt-order}
(\Delta +k^2n)\frac{1}{n-1}(\Delta+k^2)u=0 &\quad \mbox{in }\O.
\end{align}

The variational formulation of problem \eqref{eqfourt-order}
can be stated as: Find $(k,u)\in\mathbb{C}\times H_0^2(\O)$, $u\ne0$
such that
\begin{align}\label{eqfourt-order-weak}
\int_{\O} \frac{1}{n-1}(\Delta u +k^2 u)(\Delta \overline{v}+k^2n\overline{v})=0& \quad \forall v\in H_0^2(\O),
\end{align}
where $\overline{v}$ denotes the complex conjugate of $v$.
Now, expanding the previous expression we obtain
the following quadratic eigenvalue problem:
\begin{align}\label{eqfourt-ord-expand}
\int_\O \frac{1}{n-1}\Delta u\Delta \overline{v}+\tau\int_\O \frac{1}{n-1} u\Delta\overline{v}
+\tau\int_\O \frac{1}{n-1} \Delta u \overline{nv}+\tau^2 \int_\O \frac{1}{n-1} u\overline{nv}=0
\quad \forall v\in H_0^2(\O),
\end{align}
where $\tau:=k^2$.
It is easy to show that $k = 0$ is not an eigenvalue of the
problem (see \cite{CakMonkSun2014}).
Moreover, for the sake of simplicity, we will assume that
the index of refraction function $n(x)$
as a real constant. Nevertheless, this assumption
do not affect the generality of the forthcoming analysis.

For the theoretical analysis it is convenient
to transform problem \eqref{eqfourt-ord-expand}
into a linear eigenvalue problem. With this aim,
let $\phi$ be the solution of the problem:
Find $\phi\in \Hcu$ such that
\begin{align}
&\Delta \phi=\tau \frac{n}{n-1}u \quad &\mbox{in }&\O,\label{prob-aux1}\\
&\phi=0 \quad &\mbox{on }&\partial\Omega.\label{prob-aux2}
\end{align}

Therefore, by testing problem \eqref{prob-aux1}-\eqref{prob-aux2} with functions
in $H_0^1(\O)$, we arrive at the following weak formulation of the problem:
\begin{problem}\label{P1}
Find $(\lambda,u,\phi)\in \mathbb{C}\times H_0^2(\O)\times H_0^1(\O)$ with $(u,\phi)\neq 0$ such that 
\begin{align*}
a((u,\phi),(v,\psi))=\lambda b((u,\phi),(v,\psi))\quad \forall (v,\psi)\in  \HdoO\times\Hcu,
\end{align*}
\end{problem}
where $\lambda=-\tau$ and the sesquilinear forms $a(\cdot,\cdot)$
and $b(\cdot,\cdot)$ are defined by
\begin{align}
a((u,\phi),(v,\psi)):=&\frac{1}{n-1}\int_\O D^2 u:D^2 \overline{v}
+\int_\O \nabla \phi \cdot \nabla \overline{\psi},\nonumber\\[2ex] 
b((u,\phi),(v,\psi)):=&\frac{n}{n-1} \int_\O \Delta u \overline{v}+\frac{1}{n-1}\int_\O u\Delta\overline{v}
-\int_\O \nabla \phi \cdot\nabla \overline{v}+ \frac{n}{n-1}\int_\O  u\overline{\psi},\nonumber
\end{align}
for all $(u,\phi),(v,\psi)\in\HdoO\times\Hcu$. Moreover,
$":"$ denotes the usual scalar product of $2\times2$-matrices,
$D^2 u:=(\partial_{ij}u)_{1\le i,j\le2}$ denotes the Hessian matrix of $u$.

We endow $\HdoO\times\Hcu$ with the corresponding product norm,
which we will simply denote $\Vert(\cdot,\cdot)\Vert$.

Now, we note that the sesquilinear forms $a(\cdot,\cdot)$
and $b(\cdot,\cdot)$ are bounded forms.
Moreover, we have that $a(\cdot,\cdot)$ is elliptic.

\begin{lemma}
\label{ha-elipt}
There exists a constant $\alpha_{0}>0$, depending on $\O$, such that
$$
a((v,\psi),(v,\psi))
\ge\alpha_{0}\left\|(v,\psi)\right\|^2
\qquad\forall (v,\psi)\in \HdoO\times\Hcu.
$$
\end{lemma}

\begin{proof}
The result follows immediately from the fact that 
$\{\Vert D^2 v\Vert_{0,\O}^2 + \Vert \nabla\psi\Vert_{0,\O}^2\}^{1/2}$ is a norm on $\HdoO\times \Hcu$,
equivalent with the usual norm. 
\end{proof}

We define the solution operator associated with Problem~\ref{P1}:
\begin{equation*}
\begin{array}{clll}
T: \HdoO\times \Hcu & \longrightarrow& \HdoO\times\Hcu\\
 (f,g)&\longmapsto &T(f,g)=(\tilde{u},\tilde{\phi})
\end{array}
\end{equation*}
as the unique solution (as a consequence of Lemma~\ref{ha-elipt})
of the corresponding source problem:
\begin{equation}\label{source}
a((\tilde{u},\tilde{\phi}),(v,\psi))=b((f,g),(v,\psi))\quad \forall (v,\psi)\in \HdoO\times\Hcu. 
\end{equation}

The linear operator $T$ is then well defined and bounded.
Notice that $(\lambda,u,\phi)\in\mathbb{C}\times\HdoO\times\Hcu$ solves 
Problem~\ref{P1} if and only if $(\mu,u,\phi)$, with $\mu:=\frac1{\l}$,
is an eigenpair of $T$, i.e., $T(u,\phi)=\mu (u,\phi)$.

We observe that no spurious eigenvalues are introduced
into the problem since if $\mu\ne0$, $(0,\phi)$ is not an
eigenfunction of the problem.

The following is an additional regularity result for the solution of
the source problem~\eqref{source} and consequently, for the generalized eigenfunctions of $T$.

\begin{lemma}\label{lem_regul}
There exist $s,t\in(1/2,1]$ and $C>0$ such that,
for all $(f,g)\in \HdoO\times\Hcu$, the solution
$(\tilde{u},\tilde{\phi})$ of problem~\eqref{source}
satisfies $\tilde{u}\in H^{2+s}(\Omega)$, $\tilde{\phi}\in H^{1+t}(\Omega)$, and
$$\Vert\tilde{u}\Vert_{2+s,\Omega}+\Vert \tilde{\phi}\Vert_{1+t,\Omega}\le C\Vert(f,g)\Vert.$$
\end{lemma}
\begin{proof}
The estimate for $\tilde{\phi}$ follows
from the classical regularity result for the Laplace
problem with its right-hand side in $\LO$.
The estimate for $\tilde{u}$ follows from 
the classical regularity result for the biharmonic
problem with its right-hand side
in $H^{-1}(\Omega)$ (cf. \cite{G}).
\end{proof}

\begin{remark}\label{rem_regul}
The constant $s$ in the lemma above is the Sobolev regularity for the
biharmonic equation with the right-hand side in $H^{-1}(\Omega)$
and homogeneous Dirichlet boundary conditions.
The constant $t$ is the Sobolev exponent for the Laplace problem with
homogeneous Dirichlet boundary conditions.
These constants only depend on the domain $\Omega$.
If $\Omega$ is convex, then $s=t=1$. Otherwise, the lemma holds
for all $s < s_0$ and $t < t_0$, where $s_0,t_0\in(1/2,1]$
depend on the largest reentrant angle of $\Omega$.
\end{remark}

Hence, because of the compact inclusions $ H^{2+s}(\O) \hookrightarrow \HdoO$
and $H^{1+t}(\O)\hookrightarrow \Hcu$, we can conclude that $T$ is a compact operator.
So, we obtain the following spectral characterization result.

\begin{lemma} The spectrum of $T$ satisfies $\sp(T)=\{0\}\cup\{\mu_k\}_{k\in \N}$,
where  $\{\mu_k\}_{k\in \N}$ is a sequence of complex eigenvalues which converges to 0
and their corresponding eigenspaces lie in $H^{2+s}(\O)\times H^{1+t}(\O)$.
In addition $\mu=0$ is an infinite multiplicity eigenvalue of $T$.
\end{lemma}
\begin{proof}
The proof is obtained from the compactness of $T$ and Lemma~\ref{lem_regul}.
\end{proof}

\setcounter{equation}{0}
\section{The virtual element discretization}\label{SEC:DISCR_PROB}

In this section, we will write the $C^1$-VEM discretization of Problem~\ref{P1}.
With this aim, we start with the mesh construction
and the assumptions considered to introduce the discrete virtual element spaces.

Let $\left\{\CT_h\right\}_h$ be a sequence of decompositions of $\O$
into polygons $\E$ we will denote by $h_\E$ the diameter of the element $\E$
and $h$ the maximum of the diameters of all the elements of the mesh,
i.e., $h:=\max_{\E\in\CT_h}h_\E$.
In what follows, we denote by $N_\E$ the number of vertices of $\E$,
by $e$ a generic edge of $\left\{\CT_h\right\}_h$ and for all $e\in\partial\E$,
we define a unit normal vector $\nu_\E^e$ that points outside of $\E$.

In addition, we will make the following
assumptions as in \cite{BBCMMR2013,BMRR}:
there exists a positive real number $C_{\CT}$ such that,
for every $h$ and every $\E\in \CT_h$,
\begin{itemize}
\item[{\bf A1}:] the ratio between the shortest edge
and the diameter $h_\E$ of $\E$ is larger than $C_{\CT}$;
\item[{\bf A2}:] $\E\in\CT_h$ is star-shaped with
respect to every point of a  ball
of radius $C_{\CT}h_\E$.
\end{itemize}

In order to introduce the method,
we first define two preliminary discrete spaces
as follows: For each polygon $\E\in \mathcal{T}_h$
(meaning open simply connected set whose boundary is a
non-intersecting line made of a finite number of straight line segments)
we define the following finite dimensional spaces, 
\begin{align*}
\widetilde{W}_h^K
:=\left\{v_h\in \HdoK : \Delta^2v_h\in\P_{2}(\E), v_h|_{\partial\E}\in C^0(\partial\E),
v_h|_e\in\P_3(e)\,\,\forall e\in\partial\E,\right.\\
\left.\nabla v_h|_{\partial\E}\in C^0(\partial\E)^2,
\partial_\nu v_h|_e\in\P_1(e)\,\,\forall e\in\partial\E\right\},
\end{align*}
and
\begin{align*}
\widetilde{V}_h^K:=\{\psi_h\in \HuE :\Delta \psi_h \in \mathbb{P}_1(K),
\psi_h|_{\partial K}\in C^0(\partial K), \psi_h|_{e}\in \mathbb{P}_1(e) \  \forall e \in \partial K  \},
\end{align*}
where $\Delta^2$ represents the biharmonic operator
and we have denoted by $\mathbb{P}_k(S)$
the space of polynomials of degree up to $k$
defined on the subset $S\subseteq\R^2$.

The following conditions hold:
\begin{itemize}
\item for any $v_h\in\widetilde{W}_h^\E$ the trace on the boundary of $\E$
is continuous and on each edge is a polynomial of degree 3;
\item for any $v_h\in\widetilde{W}_h^\E$ the gradient on the boundary is continuous and on each edge its normal
(respectively tangential) component is a polynomial of degree 1 (respectively 2);
\item for any $\psi_h\in\widetilde{V}_h^\E$ the trace on the boundary of $\E$
is continuous and on each edge is a polynomial of degree 1;
\item $\mathbb{P}_2(\E)\times \mathbb{P}_1(\E)\subseteq \widetilde{W}_h^\E\times \widetilde{V}_h^\E$.
\end{itemize}

Next, with the aim to choose the degrees of freedom for both spaces,
we will introduce three sets ${\bf D_1}$, ${\bf D_2}$ and ${\bf D_3}$.
The first two sets (${\bf D_1}, {\bf D_2}$) are provided by linear operators from $\widetilde{W}_h^K$
into $\R$ and the set ${\bf D_3}$ by linear operators from $\widetilde{V}_h^K$
into $\R$. For all $(v_{h},\psi_h)\in\widetilde{W}_h^K\times \widetilde{V}_h^K$
they are defined as follows:

\begin{itemize}
\item ${\bf D_1}$ contains linear operators
evaluating $v_{h}$ at the $N_{\E}$ vertices of $\E$,
\item ${\bf D_2}$ contains linear operators evaluating
$\nabla v_h$ at the $N_{\E}$ vertices of $\E$,
\item ${\bf D_3}$ contains linear operators
evaluating $\psi_{h}$ at the $N_{\E}$ vertices of $\E$.
\end{itemize}

Note that, as a consequence of definition of the discrete spaces,
the output values of the three sets of operators ${\bf D_1}$, ${\bf D_2}$
and ${\bf D_3}$, are sufficient to uniquely determine $v_h$ and $\nabla v_h$
on the boundary of $\E$, and $\psi_h$ on the boundary of $\E$, respectively.

In order to construct the discrete scheme, we need some
preliminary definitions. First, we split the forms $a(\cdot,\cdot)$
and $b(\cdot,\cdot)$,
introduced in the previous section, as follows:

$$
a((u,\phi),(v,\psi))=\sum_{\E\in\CT_h}a^{\Delta}_{\E}(u,v)+a^{\nabla}_{\E}(\phi,\psi),
\qquad (u,\phi),(v,\phi)\in\HdoO\times\Hcu,
$$
$$
b((u,\phi),(v,\psi))=\sum_{\E\in\CT_h}b_{\E}((u,\phi),(v,\psi)),
\qquad (u,\phi),(v,\phi)\in\HdoO\times\Hcu,
$$
with
\begin{equation*}
\label{alocal}
a^{\Delta}_{\E}(u,v)
:=\int_{\E}D^2 u:\,D^2 \overline{v},
\qquad u,v\in\HdoK,
\end{equation*}
\begin{equation*}
\label{alocalnabla}
a^{\nabla}_{\E}(\phi,\psi)
:=\int_{\E}\nabla \phi\cdot\nabla\overline{\psi},
\qquad \phi,\psi\in H^1(K),
\end{equation*}
and for all $(u,\phi),(v,\phi)\in H^2(K)\times H^1(K)$,
\begin{equation*}
\label{blocal}
b_{\E}((u,\phi),(v,\psi))
:=\frac{n}{n-1} \int_\E \Delta u \overline{v}+\frac{1}{n-1}\int_\E u\Delta\overline{v}
-\int_\E \nabla \phi \cdot\nabla \overline{v}+ \frac{n}{n-1}\int_\E  u\overline{\psi}.
\end{equation*}
Now, we define the projector
$\Pi_{2}^{\Delta}:\ H^2(\E)\longrightarrow\P_2(\E)\subseteq\widetilde{W}_h^K$
for each $v\in H^2(\E)$ as the solution of 
\begin{subequations}
\begin{align}
a^\Delta_{\E}\big(\Pi_{2}^{\Delta} v,q\big)
& =a^\Delta_{\E}(v,q)
\qquad\forall q\in\P_2(\E),
\label{numero}
\\
((\Pi_{2}^{\Delta} v,q))_{\E}
&=((v,q))_{\E} \qquad\forall q\in\P_1(\E),
\label{numeroo}
\end{align}
\end{subequations}
where $((\cdot,\cdot))_{\E}$ is defined as follows:
\begin{equation*}
((u,v))_{\E}=\sum_{i=1}^{N_{\E}}u(P_i)v(P_i)\qquad\forall u,v\in C^0(\partial\E),
\end{equation*}
where $P_i, 1\le i\le N_{\E}$, are the vertices of $\E$.
We note that the bilinear form $a^\Delta_{\E}(\cdot,\cdot)$
has a non-trivial kernel, given by $\P_1(\E)$. Hence,
the role of condition \eqref{numeroo} is to select
an element of the kernel of the operator.
We observe that operator $\Pi_{2}^{\Delta}$ is well defined on $\widetilde{W}_h^K$
and, most important, for all $v\in\widetilde{W}_h^K$ the polynomial
$\Pi_{2}^{\Delta} v$ can be computed using only the values of the operators
${\bf D_1}$ and ${\bf D_2}$
calculated on $v$. This follows easily with an integration by parts
(see \cite{ABSV2016}). 

In a similar way, we define the projector
$\Pi_1^{\nabla}:\ H^1(\E)\longrightarrow\P_1(\E)\subseteq\widetilde{V}_h^K$ for
each $\psi\in H^1(\E)$ as the solution of 
\begin{subequations}
\begin{align}
a^\nabla_{\E}\big(\Pi^{\nabla}_1 \psi,q\big)
& =a^\nabla_{\E}(\psi,q)
\qquad\forall q\in\P_1(\E),
\label{numero1}
\\
(\Pi^{\nabla}_1\psi,1)_{\partial\E}
&=(\psi,1)_{\partial\E}.
\label{numeroo1}
\end{align}
\end{subequations}
We observe that operator $\Pi^{\nabla}_1$ is well defined on $\widetilde{V}_h^K$
and, as before, for all $\psi\in\widetilde{V}_h^K$ the polynomial
$\Pi^{\nabla}_1\psi$ can be computed using only the values of the operators
${\bf D_3}$ calculated on $\psi$, which follows by an integration by parts
(see \cite{AABMR13}). 

Now, we introduce our local virtual spaces: 
\begin{align*}
\WK:=\left\{v_h\in\widetilde{W}_h^K : \int_{\E}(\Pi_2^{\Delta} v_h)q=\int_{\E}v_hq\qquad\forall q\in\P_{2}(\E)\right\},
\end{align*}
and
\begin{align*}
\VK:=\left\{ \psi_h \in \widetilde{V}_h^K: \int_K (\Pi_1^{\nabla}\psi_h)q=\int_K \psi_h q
\qquad\forall q\in \P_1(\E) \right\}.
\end{align*}

It is clear that $W_h^K\times V_h^K\subseteq \widetilde{W}_h^K\times \widetilde{V}_h^K$.
Thus, the linear operators $\Pi_2^{\Delta}$ and $\Pi_1^{\nabla}$
are well defined on $W_h^K$ and $V_h^K$, respectively.

In \cite[Lemma~2.1]{ABSV2016} has been established that the
sets of operators ${\bf D_1}$ and ${\bf D_2}$ constitutes a
set of degrees of freedom for the space $\WK$.
Moreover, the set of operators ${\bf D_3}$ constitutes a
set of degrees of freedom for the space $\VK$ (see \cite{AABMR13}).

We also have that $\mathbb{P}_2(K)\times\mathbb{P}_1(K)\subseteq\WK\times\VK$.
This will guarantee the good approximation properties for the spaces.

To continue the construction of the discrete scheme,
we will need to consider new projectors:
First, we define the projector
$\Pi_2^{\nabla}:\ H^2(\E)\longrightarrow\P_2(\E)$ for
each $w\in H^2(\E)$ as the solution of 
\begin{subequations}
\begin{align}
a^\nabla_{\E}\big(\Pi^{\nabla}_2 w,q\big)
& =a^\nabla_{\E}(w,q)
\qquad\forall q\in\P_2(\E),
\label{numero2}
\\
(\Pi^{\nabla}_2 w,1)_{0,\E}
&=(w,1)_{0,\E}.
\label{numeroo2}
\end{align}
\end{subequations}
Moreover, we consider the $\LO$ orthogonal projectors
onto $\P_{l}(\E)$, $l=1,2$ as follows:
we define $\Pi_{l}^{0}:\LO\to\P_{l}(\E)$ for each $p \in \LO$ by
\begin{equation}\label{fff}
\int_{\E}(\Pi_{l}^{0} p)q=\int_{\E}pq\qquad\forall q\in\P_{l}(\E).
\end{equation}

Now, due to the particular property appearing in definition
of the space $\WK$, it can be seen that the right hand
side in \eqref{fff} is computable using $\Pi_{2}^{\Delta} v_h$,
and thus $\Pi_{2}^{0} v_h$ depends only on the values of
the degrees of freedom for $v_h$ and $\nabla v_h$.
Actually, it is easy to check that on the space
$\WK$ the projectors $\Pi_{2}^{0}v_h$ and $\Pi_{2}^{\Delta} v_h$
are the same operator. In fact:
\begin{equation}\label{L2proj}
\int_{\E}(\Pi_{2}^{0} v_h)q=\int_{\E}v_hq=
\int_{\E}(\Pi_{2}^{\Delta} v_h)q\qquad\forall q\in\P_{2}(\E).
\end{equation}
Repeating the arguments, it can be proved that
$\Pi_{1}^{0}\phi_h$ and $\Pi_{1}^{\nabla}\phi_h$
are the same operator in $\VK$.

Now, for every decomposition $\mathcal{T}_h$ of $\Omega$ into simple polygons $K$,
 we introduce our the global virtual space denoted by $Z_h$ as follow: $$Z_h:=W_h\times V_h,$$
 where
\begin{align*}
W_h:=\{v_h\in \Hcd: v_h|_{K}\in \WK\} \quad \mbox{and}\quad  
V_h:= \{\psi_h\in \Hcu: \psi_h|_{K}\in \VK\}.
\end{align*}

A set of degrees of freedom for $Z_h$ is given
by all pointwise values of $v_h$ and $\psi_h$ on all vertices of $\mathcal{T}_h$
together with all pointwise values of $\nabla v_h$ on all vertices
of $\mathcal{T}_h$, excluding the vertices on $\partial\Omega$
(where the values vanishes). Thus, the dimension
of $Z_h$ is four times the number of interior vertices of $\mathcal{T}_h$.

In what follows, we discuss the construction
of the discrete version of the local forms.
With this aim, we consider $s_{K}^{\Delta}(\cdot,\cdot)$
and $s_{K}^{\nabla}(\cdot,\cdot)$ any symmetric positive definite
forms satisfying: 
\begin{align}
c_0a_K^{\Delta}(v_h,v_h)\leq s_\E^{\Delta}(v_h,v_h)\leq c_1
a_K^{\Delta}(v_h,v_h)&\quad \forall v_h \in W_h^K\quad \mbox{with }\quad
\Pi_2^{\Delta} v_h =0 \label{term-stab-SK},\\[1ex]
c_2a_K^{\nabla}(\psi_h,\psi_h)\leq s_\E^{\nabla}(\psi_h,\psi_h)\leq
c_3 a_K^{\nabla}(\psi_h,\psi_h)&\quad \forall \psi_h \in V_h^K\quad
\mbox{with }\quad \Pi_1^{\nabla} \psi_h =0. \label{term-stab-SK0}
\end{align}

We define the discrete sesquilinear forms $a_h(\cdot,\cdot):Z_h\times Z_h\to \mathbb{C} $
and $b_h(\cdot,\cdot):Z_h\times Z_h\to \mathbb{C}$ by
\begin{align*}
a_h((u_h,\phi_h),(v_h,\psi_h))&
:=\sum_{\E\in\CT_h} a_{h,\E}^{\Delta}(u_h,v_h)+a_{h,\E}^{\nabla}(\phi_h,\psi_h)
\qquad \forall (u_h,\phi_h),(v_h,\psi_h)\in Z_h,\\
b_h((u_h,\phi_h),(v_h,\psi_h))&
:=\sum_{\E\in\CT_h}b_{h,\E}((u_h,\phi_h),(v_h,\psi_h))
\qquad \forall  (u_h,\phi_h),(v_h,\psi_h)\in Z_h,
\end{align*}
where $a_{h,\E}^{\Delta}(\cdot,\cdot)$,
$a_{h,\E}^{\nabla}(\cdot,\cdot)$ and $b_{h,\E}(\cdot,\cdot)$
are local forms on $\WK\times\WK$ and $\VK\times\VK$
defined by
\begin{eqnarray}
a_{h,\E}^{\Delta}(u_h,v_h)
:=a_{\E}^{\Delta}\big(\Pi_2^{\Delta} u_h,\Pi_2^{\Delta} v_h\big)
+s_{\E}^{\Delta}\big(u_h-\Pi_2^{\Delta} u_h,v_h-\Pi_2^{\Delta} v_h\big),
\qquad \forall u_h,v_h\in\WK,\nonumber\\
a_{h,\E}^{\nabla}(\phi_h,\psi_h)
:=a_{\E}^{\nabla}\big(\Pi_1^{\nabla} \phi_h,\Pi_1^{\nabla} \psi_h\big)
+s_{\E}^{\nabla}\big(\phi_h-\Pi_1^{\nabla} \phi_h,\psi_h-\Pi_1^{\nabla} \psi_h\big),
\qquad \forall \phi_h, \psi_h\in V_h^K,\nonumber
\end{eqnarray}
\begin{align*}
b_{h,\E}((u_h,\phi_h),(v_h,\psi_h)):&= \frac{n}{n-1}\int_K \Pi_2^{0}(\Delta u_h)  \Pi_2^{0} v_h
+\frac{1}{n-1}\int_K\Pi_2^{0} u_h \Pi_2^{0}(\Delta v_h)
- \int_K \nabla \Pi_1^{\nabla}\phi_h\cdot \nabla \Pi_2^{\nabla} v_h\\
&+ \frac{n}{n-1}\int_K \Pi_2^{0} u_h \Pi_1^{0}\psi_h 
\qquad \forall (u_h,\phi_h),(v_h,\psi_h)\in\WK\times\VK.
\end{align*}

The construction of the local sesquilinear forms
guarantees the usual consistency and stability
properties, as is stated in the proposition below.
Since the proof follows standard arguments in the VEM
literature, it is omitted.
\begin{proposition}
The local forms $a_{h,\E}^{\Delta}(\cdot,\cdot)$ and
$a_{h,\E}^{\nabla}(\cdot,\cdot)$
on each element $\E$ satisfy 
\begin{itemize}
\item \textit{Consistency}: for all $h > 0$ and for all $\E\in\CT_h$ we have that
\begin{align}
a_{h,\E}^{\Delta}(v_h,q)=a_{\E}^{\Delta}(v_h,q)\qquad\forall q \in\P_2(\E),
\quad\forall v_h\in \WK,\label{consis-a1}\\
a_{h,\E}^{\nabla}(\psi_h,q)=a_{\E}^{\nabla}(\psi_h,q)\qquad\forall q \in\P_1(\E),
\quad\forall \psi_h\in \VK.\label{consis-a2}
\end{align}
\item \textit{Stability and boundedness}: There exist positive constants
$\alpha_i, i=1,2,3,4,$ independent of $\E$, such that:
\begin{align}
\alpha_1 a_{\E}^{\Delta}(v_h,v_h)\leq a_{h,\E}^{\Delta}(v_h,v_h)
\leq\alpha_2 a_{\E}^{\Delta}(v_h,v_h)\qquad\forall v_h\in\WK,\label{stab-a1}\\
\alpha_3 a_{\E}^{\nabla}(\psi_h,\psi_h)\leq a_{h,\E}^{\nabla}(\psi_h,\psi_h)
\leq\alpha_4 a_{\E}^{\nabla}(\psi_h,\psi_h)\qquad\forall \psi_h\in\VK.\label{stab-a2}
\end{align}
\end{itemize}
\end{proposition}

Now, we are in a position to write the virtual
element discretization of Problem~\ref{P1}.

\begin{problem}\label{P2}
Find $(\lambda_h,u_h,\psi_h)\in\mathbb{C}\times Z_h$,
$(u_h,\phi_h)\ne0$ such that
\begin{align}
a_h((u_h,\phi_h),(v_h,\psi_h))=\lambda_hb_h((u_h,\phi_h),(v_h,\psi_h)).\label{probspecdiscr}
\end{align}
\end{problem}

It is clear that by virtue of \eqref{stab-a1} and \eqref{stab-a2} the
sesquilinear form $a_{h}(\cdot,\cdot)$ is bounded.
Moreover, we will show in the following lemma that
$a_{h}(\cdot,\cdot)$ is also uniformly elliptic.
\begin{lemma}
\label{ha-elipt-disc}
There exists a constant $\beta>0$, independent of $h$, such that
$$
a_{h}((v_h,\psi_h),(v_h,\psi_h))
\ge\beta \Vert(v_h,\psi_h)\Vert^2\qquad\forall (v_h,\psi_h)\in Z_h.
$$
\end{lemma}
\begin{proof}
The result is deduced from Lemma~\ref{ha-elipt}, \eqref{stab-a1} and \eqref{stab-a2}.
\end{proof}

Now, we introduce the discrete solution operator $T_h$ which is given by
\begin{equation*}
\begin{array}{cccll}
T_h:& \Hcd\times \Hcu     & \longrightarrow & \Hcd\times \Hcu&\\
  &(f,g)& \longmapsto     & T_h(f,g)=(\tilde{u}_h,\tilde{\phi}_h)&
\end{array}
\end{equation*}
where $(\tilde{u}_h,\tilde{\phi}_h)\in  Z_h$
is the unique solution of the corresponding
discrete source problem
\begin{equation}\label{source-disc}
a_h((\tilde{u}_h,\tilde{\phi}_h),(v_h,\psi_h))=b_{h}((f,g),(v_h,\psi_h))
\qquad\forall (v_h,\psi_h)\in Z_h.
\end{equation}

Because of Lemma~\ref{ha-elipt-disc}, the linear operator $T_h$
is well defined and bounded uniformly with respect to $h$.
Once more, as in the continuous case,
$(\lambda_h,u_h,\phi_h)\in\mathbb{C}\times Z_h$ solves 
Problem~\ref{P2} if and only if $(\mu_h,u_h,\phi_h)$, with $\mu_h:=\frac1{\l_h}$,
is an eigenpair of $T_h$, i.e., $T_h(u_h,\phi_h)=\mu_h (u_h,\phi_h)$.

\setcounter{equation}{0}
\section{Spectral approximation and error estimates}
\label{SEC:approximation}

To prove that $T_h$ provides a correct spectral approximation
of $T$, we will resort to the classical
theory for compact operators (see \cite{BO}).
With this aim, we first recall the following approximation result
which is derived by interpolation
between Sobolev spaces (see for instance \cite[Theorem I.1.4]{GR}
from the analogous result for integer values of $s$).
In its turn, the result for integer values is stated
in \cite[Proposition 4.2]{BBCMMR2013} and follows from the
classical Scott-Dupont theory (see \cite{BS-2008}
and \cite[Proposition 3.1]{ABSV2016}):
\begin{proposition}
\label{app1}
There exists a constant $C>0$,
such that for every $v\in H^{\delta}(\E)$ there exists
$v_{\pi}\in\P_k(\E)$, $k\geq 0$ such that
 \begin{eqnarray*}
\vert v-v_{\pi}\vert_{\ell,\E}\leq C h_\E^{\delta-\ell}|v|_{\delta,\E}\quad 0\leq\delta\leq
k+1, \ell=0,\ldots,[\delta],
\end{eqnarray*}
with $[\delta]$ denoting largest integer equal or smaller than $\delta \in {\mathbb R}$.
\end{proposition}

For the analysis we will introduce some broken seminorms:
$$|\psi|_{1,h}^{2}:=\sum_{\E\in\CT_h}|\psi|_{1,\E}^{2}  \qquad \mbox{and } \quad
|v|_{2,h}^{2}:=\sum_{\E\in\CT_h}|v|_{2,\E}^{2}, $$
which are well defined for every $(\psi,v)\in [L^{2}(\O)]^2$ such that
$(\psi,v)|_{\E}\in H^{1}(\E)\times H^{2}(\E)$ for all polygon $\E\in \CT_{h}$.

In what follows, we derive several auxiliary results which will be used
in the following to prove convergence and error estimates for the spectral approximation.

\begin{proposition}\label{app0}
Assume {\textbf{A1}--\textbf{A2}} are satisfied,
let $\psi\in\HutO$ with $t\in(0,1]$. Then, there exist $\psi_{I}\in\Vh$
and $C>0$ such that
$$\Vert\psi-\psi_{I}\Vert_{1,\O}\le Ch^t\vert\psi\vert_{1+t,\O}.$$
\end{proposition}
\begin{proof}
This result has been proved in \cite[Theorem~11]{CGPS} (see also \cite[Proposition~4.2]{MRR2015}).
\end{proof}

\begin{proposition}\label{app2}
Assume {\textbf{A1}--\textbf{A2}} are satisfied,
let $v\in\HdsO$ with $s\in(0,1]$. Then, there exist $v_{I}\in\Wh$
and $C>0$ such that
$$\Vert v-v_{I}\Vert_{2,\O}\le Ch^s\vert v\vert_{2+s,\O}.$$
\end{proposition}
\begin{proof}
This result has been establish in \cite[Proposition~3.1]{ABSV2016}.
\end{proof}

Now, we establish a result which will be
useful to prove the convergence of the operator
$T_h$ to $T$ as $h$ goes to zero.

\begin{lemma}
\label{lemcotste}
There exists $C>0$ independent of $h$ such that
for all $(f,g)\in \Hcd\times\Hcu$, if $(\tilde{u},\tilde{\phi}):=T(f,g)$ and
$(\tilde{u}_h,\tilde{\phi}_h):=T_h(f,g)$, then
\begin{align*}
\left\|\left(T-T_h\right)(f,g)\right\|\leq C  h \Vert(f,g)\Vert + \vert\tilde{u}-\tilde{u}_{I}\vert_{2,\O}
+\vert \tilde{u}-\tilde{u}_{\pi} \vert_{2,h} + \vert\tilde{\phi} - \tilde{\phi}_{I}\vert_{1,\O}
+\vert \tilde{\phi} - \tilde{\phi}_{\pi}\vert_{1,h},
\end{align*}
for all $(\tilde{u}_I,\tilde{\phi}_I)\in Z_h$
and for all $(\tilde{u}_{\pi},\tilde{\phi}_{\pi})\in[\LO]^2$
such that $(\tilde{u}_{\pi},\tilde{\phi}_{\pi})|_K\in \mathbb{P}_2(K)\times \mathbb{P}_1(K)$.
\end{lemma}
\begin{proof}
Let $(f,g)\in \Hcd\times\Hcu$, for any $(\tilde{u}_I,\tilde{\phi}_I)\in\Wh\times\Vh$, we have,
\begin{align}
\Vert(T-T_h)(f,g)\Vert \leq \Vert(\tilde{u},\tilde{\phi})-(\tilde{u}_I,\tilde{\phi}_I)\Vert
+\Vert(\tilde{u}_I,\tilde{\phi}_I)-(\tilde{u}_h,\tilde{\phi}_h)\Vert.\label{estim-oper}
\end{align}
Now, we define $(v_h,\psi_h)=(\tilde{u}_h-\tilde{u}_I,\tilde{\phi}_h-\tilde{\phi}_I)\in Z_h$,
then from the ellipticity of $a_h(\cdot,\cdot)$ 
and the definition of $T$ and $T_h$, we have
\begin{align}
&\beta\Vert(v_h,\psi_h)\Vert^2 \leq a_h((v_h,\psi_h),(v_h,\psi_h))
=a_h((\tilde{u}_h,\tilde{\phi}_h),(v_h,\psi_h))-a_h((\tilde{u}_I,\tilde{\phi}_I),(v_h,\psi_h))\nonumber\\[1ex]
&=b_h((f,g),(v_h,\psi_h))-\sumkth\Big\{a_{h,K}^{\Delta}(\tilde{u}_I,v_h)
+a_{h,K}^{\nabla}(\tilde{\phi}_I,\psi_h)\Big\}\nonumber\\[1ex]
&=b_h((f,g),(v_h,\psi_h))-\sumkth\Big\{a_{h,K}^{\Delta}(\tilde{u}_I-\tilde{u}_{\pi},v_h)+a_{h,K}^{\Delta}(\tilde{u}_{\pi},v_h)
+a_{h,K}^{\nabla}(\tilde{\phi}_I-\tilde{\phi}_{\pi},\psi_h)+a_{h,K}^{\nabla}(\tilde{\phi}_{\pi},\psi_h) \Big\}\nonumber\\[1ex]
&=b_h((f,g),(v_h,\psi_h))-\sumkth\Big\{a_{h,K}^{\Delta}(\tilde{u}_I-\tilde{u}_{\pi},v_h)
+a_{K}^{\Delta}(\tilde{u}_{\pi}-\tilde{u},v_h)+a_{K}^{\Delta}(\tilde{u},v_h)\nonumber\\[1ex]
&\hspace{5.0cm}+a_{h,K}^{\nabla}(\tilde{\phi}_I-\tilde{\phi}_{\pi},\psi_h)
+a_{K}^{\nabla}(\tilde{\phi}_{\pi}-\tilde{\phi},\psi_h)+a_{K}^{\nabla}(\tilde{\phi},\psi_h) \Big\}\nonumber\\[1ex]
&=\underbrace{\sumkth \Big\{ b_{h,K}((f,g ),(v_h,\psi_h))-b_K((f,g),(v_h,\psi_h))\Big\}}_{E_1}
-\underbrace{\sumkth\Big\{a_{h,K}^{\Delta}(\tilde{u}_I-\tilde{u}_{\pi},v_h)
+a_{K}^{\Delta}(\tilde{u}_{\pi}-\tilde{u},v_h)\Big\}}_{E_2}\nonumber \\[1ex]
&- \underbrace{\sumkth \Big\{ a_{h,K}^{\nabla}(\tilde{\phi}_I-\tilde{\phi}_{\pi},\psi_h)
+a_{K}^{\nabla}(\tilde{\phi}_{\pi}-\tilde{\phi},\psi_h) \Big\}}_{E_3}\label{cotconv},
\end{align}
where we have used the consistency properties \eqref{consis-a1}-\eqref{consis-a2}.
We now bound each term $E_i|_K$, $i = 1,2,3$.

First, the term $E_1|_K$ can be written as follows:
\begin{align}
& b_{h,K}((f,g),(v_h,\psi_h))-b_K((f,g),(v_h,\psi_h))\nonumber\\[1ex]
&=\frac{n}{n-1}\Big\{\underbrace{\int_K \Pi_2^{0}(\Delta f)\Pi_2^{0}v_h-\int_K \Delta f \overline{v_h}}_{E_{11}} \Big\}
+\frac{1}{n-1}\Big\{\underbrace{\int_K \Pi_2^{0}f \Pi_2^{0}(\Delta v_h)
-\int_K f \Delta \overline{v_h}}_{E_{12}} \Big\}\nonumber\\[1ex]
& -\Big\{\underbrace{\int_K \nabla \Pi_1^{\nabla}g \cdot \nabla \Pi_2^{\nabla}v_h
- \int_K \nabla g \cdot \nabla \overline{v_h}}_{E_{13}} \Big\}
+ \frac{n}{n-1}\Big\{\underbrace{\int_K (\Pi_2^{0}f)(\Pi_1^{0}\psi_h)
-\int_K f \overline{\psi_h}}_{E_{14}} \Big\}.\label{exp-E_1}
\end{align}

Now, we will bound each term $E_{1i}|_K \ i=1,2,3,4$.
The term $E_{11}$ can be bounded as follows:
Using the definition of $\Pi_2^{0}$ and
Proposition~\ref{app1}, we have
\begin{align*}
E_{11}&=\int_K  \Delta f(\overline{v_h}-\Pi_2^{0}v_h)\leq |f |_{2,K} ||v_h-\Pi_2^{0}v_h||_{0,K}\\
&=|f |_{2,K} \inf_{q\in \mathbb{P}_2(K)}||v_h-q||_{0,K}\leq Ch_\E^{2}|f|_{2,K}|v_h|_{2,K}.
\end{align*}

For the term $E_{12}$, we repeat the same arguments to obtain:
\begin{align*}
E_{12}\leq Ch_K^2|f |_{2,K}|v_h|_{2,K}.
\end{align*}

Now, we bound $E_{13}$. From the definition of $\Pi_2^{\nabla}$, we have
\begin{align*}
E_{13}&=\int_K \nabla \Pi_1^{\nabla}g \cdot \nabla \overline{v_h}
- \int_K \nabla g \cdot \nabla \overline{v_h}=\int_K \nabla(\Pi_1^{\nabla}g -g )\cdot \nabla \overline{v_h}\\[1ex]
&=\int_K\nabla(\Pi_1^{\nabla}g -g )\cdot \nabla (\overline{v_h} -\tilde{v}_{\pi})\leq
|\Pi_1^{\nabla}g -g  |_{1,K}|v_h -\tilde{v}_{\pi} |_{1,K}\\[1ex]
&\leq Ch_K|g  |_{1,K} |v_h |_{2,K},
\end{align*}
where we have used the definition and the stability of $\Pi_1^{\nabla}$
with $\tilde{v}_{\pi}\in \mathbb{P}_1(K)$ such that
Proposition~\ref{app1} holds true.

For the term $E_{14}$, we first use the definition of $\Pi_2^{0}$,
the definition and the stability of $\Pi_1^{0}$
with respect to $\hat{f}_{\pi}\in \mathbb{P}_1(K)$ such that
Proposition~\ref{app1} holds true, thus, we have
\begin{align*}
E_{14}&=\int_K f \Pi_1^{0}\psi_h-\int_K f\overline{\psi_h}=\int_K(f-\hat{f}_{\pi})(\Pi_1^{0}\psi_h-\overline{\psi_h})\\[1ex]
&\leq Ch_K^{2}|f |_{2,K}||\Pi_1^{0}\psi_h -\psi_h ||_{0,K}
\leq Ch_K^{2} | f|_{2,K}\Vert\psi_h\Vert_{0,\E}.
\end{align*}

Therefore, using the Cauchy-Schwarz inequality, we can deduce from \eqref{exp-E_1} that
\begin{align*}
E_1\leq Ch||(f,g ) || ||(v_h,\psi_h) ||.
\end{align*}

Finally, from \eqref{cotconv} we have
\begin{align*}
\beta ||(v_h,\psi_h) ||\leq C \Big\{ h||(f,g ) ||  +| u-u_I|_{2,\O} + | u-u_{\pi}|_{2,h}
+| \phi-\phi_I|_{1,\O} + |\phi-\phi_{\pi} |_{1,h}  \Big\}.
\end{align*}
Therefore, the proof follows from \eqref{estim-oper} and the previous inequality.
\end{proof}

For the convergence and error analysis of the proposed virtual element
scheme for the transmission eigenvalue problem, we first
establish that $T_h\to T$ in norm as $h\to0$.
Then, we prove a similar convergence result
for the adjoint operators $T^*$ and $T_h^*$
of $T$ and $T_h$, respectively.

\begin{lemma}\label{propP1}
There exist $C>0$ and $\tilde{s}\in(0,1]$, independent of $h$, such that
$$\|T-T_h\|\le Ch^{\tilde{s}}.$$
\end{lemma}

\begin{proof}
Let $(f,g)\in\Hcd\times\Hcu$ such that $||(f,g)||=1$,
let $(\tilde{u},\tilde{\phi})$ and $(\tilde{u}_h,\tilde{\phi}_h)$
be the solution of problems \eqref{source} and \eqref{source-disc},
respectively, so that $(\tilde{u},\tilde{\phi}):=T(f,g)$ and
$(\tilde{u}_h,\tilde{\phi}_h):=T_h(f,g)$.
From Lemma~\ref{lemcotste}, we have
\begin{align*}
\left\|\left(T-T_h\right)(f,g )\right\|& \leq C  h ||(f,g )|| + \Vert u-u_{I}\Vert_{2,\O}
+\vert u-u_{\pi} \vert_{2,h} + \Vert \phi - \phi_{I}\Vert_{1,\O}
+\vert \phi - \phi_{\pi}\vert_{1,h}\\[1ex]
&\leq C\left(  h||(f,g ) ||+ h^s||f||_{2,\O} + h^t||g ||_{1,\O}\right)\\
&\le Ch^{\tilde{s}}|| (f,g )||\nonumber
\end{align*}
where we have used the Propositions~\ref{app1}, \ref{app0} and \ref{app2},
and  Lemma~\ref{lem_regul}, with $\tilde{s}:=\min\{s,t\}$.
Thus, we conclude the proof.
\end{proof}

Let $T^*$ and $T_h^*:\Hcd\times\Hcu\to\Hcd\times\Hcu$
the adjoint operators of $T$ and $T_h$, respectively,
defined by $T^*(f,g):=(\tilde{u}^*,\tilde{\phi}^*)$
and $T^*_h(f,g):=(\tilde{u}_h^*,\tilde{\phi}_h^*)$,
where $(\tilde{u}^*,\tilde{\phi}^*)$ and $(\tilde{u}_h^*,\tilde{\phi}_h^*)$ are
the unique solutions of the following problems:
\begin{align}
& a((v,\psi),(\tilde{u}^*,\tilde{\phi}^*))=b((v,\psi),(f,g))\qquad \forall
(v,\psi)\in \Hcd\times\Hcu,\label{AdjCont}\\[1ex]
& a_h((v_h,\psi_h),(\tilde{u}_h^*,\tilde{\phi}_h^*))=b_h((v_h,\psi_h),(f,g))\qquad
\forall (v_h,\psi_h)\in Z_h.\label{AdjDisc} 
\end{align}

It is simple to prove that if $\mu$ is an eigenvalue
of $T$ with multiplicity $m$, $\overline{\mu}$ is an eigenvalue
of $T^{*}$ with the same multiplicity $m$.

Now, we will study the convergence
in norm $T_h^{*}$ to $T^{*}$ as $h$ goes to zero.
With this aim, first we establish
an additional regularity result
for the solution $(\tilde{u}^*,\tilde{\phi}^*)$
of problem~\eqref{AdjCont}.

\begin{lemma}\label{lem_regul_adj}
There exist $s,t\in(1/2,1]$ and $C>0$ such that,
for all $(f,g)\in \HdoO\times\Hcu$, the solution
$(\tilde{u}^*,\tilde{\phi}^*)$ of problem~\eqref{AdjCont}
satisfies $\tilde{u}^*\in H^{2+s}(\Omega)$, $\tilde{\phi}^*\in H^{1+t}(\Omega)$, and
$$\Vert\tilde{u}^*\Vert_{2+s,\Omega}+\Vert\tilde{\phi}^*\Vert_{1+t,\Omega}\le C\Vert(f,g)\Vert.$$
\end{lemma}
\begin{proof}
The result follows repeating the same arguments
used in the proof of Lemma~\ref{lem_regul}.
\end{proof}

\begin{remark}
We note that the constants $s$ and $t$
in Lemma~\ref{lem_regul_adj} are the same
as in Lemma~\ref{lem_regul}.
\end{remark}

Now, we are in a position to establish the following result.
\begin{lemma}\label{propP1adj}
There exist $C>0$ and $\tilde{s}\in(0,1]$, independent of $h$, such that
$$\|T^*-T_h^*\|\le Ch^{\tilde{s}}.$$
\end{lemma}
\begin{proof}
It is essentially identical to that of Lemma~\ref{lemcotste}.
\end{proof}

Our final goal is to show convergence
and obtain error estimates.
With this aim, 
we will apply to our problem the theory from
\cite{BO,Osborn1975} for non-selfadjoint compact operators.

We first recall the definition of spectral projectors.
Let $\mu$ be a nonzero eigenvalue of $T$ with algebraic
multiplicity $m$ and let $\Gamma$ be an open disk in the complex plane
centered at $\mu$, such that
$\mu$ is the only eigenvalue of $T$ lying in $\Gamma$
and $\partial \Gamma\cap\sp(T)=\emptyset$. 
The spectral projectors $E$ and $E^{*}$ are defined as follows:
\begin{itemize}
\item The spectral projector of $T$ relative to $\mu$:
$E:=(2\pi i)^{-1}\int_{\partial \Gamma} (z-T)^{-1}dz$;
\item The spectral projector of $T^{*}$ relative to $\overline{\mu}$:
$E^*:=(2\pi i)^{-1}\int_{\partial \Gamma} (z-T^*)^{-1}dz.$
\end{itemize}
$E$ and $E^{*}$ are projections onto the space of generalized
eigenvectors $R(E)$ and $R(E^*)$, respectively.
It is simple to prove that $R(E),R(E^*)\in H^{2+s}(\Omega)\times H^{1+t}(\Omega)$.

Now, since $T_h\to T$ in norm, there exist $m$ eigenvalues (which lie in $\Gamma$)
$\mu_h^{(1)},\ldots,\mu_h^{(m)}$ of $T_h$
(repeated according to their respective multiplicities)
will converge to $\mu$ as $h$ goes to zero.

In a similar way, we introduce the following spectral
projector $E_h:=(2\pi i)^{-1}\int_{\partial \Gamma} (z-T_h)^{-1}dz$,
which is a projector onto the invariant subspace $R(E_h)$
of $T_h$ spanned by the generalized eigenvectors
of $T_h$ corresponding to $\mu_h^{(1)},\ldots,\mu_h^{(m)}$.

We recall the definition of the \textit{gap} $\hdel$ between two closed
subspaces $\mathcal{X}$ and $\mathcal{Y}$ of a Hilbert space $\mathcal{V}$:
$$
\hdel(\mathcal{X},\mathcal{Y})
:=\max\left\{\delta(\mathcal{X},\mathcal{Y}),\delta(\mathcal{Y},\mathcal{X})\right\},$$
where
$$
\delta(\mathcal{X},\mathcal{Y})
:=\sup_{\mathbf{x}\in\mathcal{X}:\ 
\left\|x\right\|_{\mathcal{V}}=1}\delta(x,\mathcal{Y}),
\quad\text{with }\delta(x,\mathcal{Y}):=
\inf_{y\in\mathcal{Y}}\|x-y\|_{\mathcal{V}}.$$

Let $\mathcal{P}_h:=\mathcal{P}_h^2\times \mathcal{P}_h^1:\Hcd\times\Hcu\to
Z_h\subseteq \Hcd\times\Hcu$ be the projector defined by 
\begin{equation*}
a(\mathcal{P}_h(u,\phi)-(u,\phi),(v_h,\psi_h))=a^{\Delta}(\mathcal{P}_h^2u-u,v_h)
+a^{\nabla}(\mathcal{P}_h^1\phi-\phi,\psi_h)=0\qquad \forall  (v_h,\psi_h)\in Z_h.
\end{equation*}
We note that the form $a(\cdot,\cdot)$ is the inner product of $\Hcd\times\Hcu$.
Therefore, we have
\begin{align}
|(u,\phi)-\mathcal{P}(u,\phi)|_{\Hcd\times\Hcu}=\inf\limits_{(v_h,\psi_h)\in Z_h}|(u,\phi)
-(v_h,\psi_h)|_{\Hcd\times\Hcu},\label{best_appr_of_P}
\end{align}
and
\begin{align}
|\mathcal{P}(u,\phi)|_{\Hcd\times\Hcu}\leq |(u,\phi)|_{\Hcd\times\Hcu}\qquad
\forall (u,\phi)\in \Hcd\times\Hcu.\label{acot_de_P}
\end{align}

The following error estimates for the approximation of eigenvalues and
eigenfunctions hold true.

\begin{theorem}
\label{gap}
There exists a strictly positive constant $C$ such that
\begin{align}
\hdel(R(E),R(E_h)) 
& \le C h^{\min\{s,t\}},\label{bound1}
\\
\left|\mu-\hat{\mu}_h\right|
& \le Ch^{2\min\{s,t\}},\label{bound2}
\end{align}
where
$\hat{\mu}_h:=\frac{1}{m}\sum\limits_{k=1}^{m}\mu_h^{(k)}$ and
with the constants $s$ and $t$ as in Lemmas~\ref{lem_regul} and \ref{lem_regul_adj}
(see also Remark~\ref{rem_regul}).
\end{theorem}
\begin{proof}
As a consequence of Lemma~\ref{propP1}, $T_h$ converges in norm to $T$
as $h$ goes to zero. Then, the proof of \eqref{bound1} follows as
a direct consequence of Theorem 7.1 from \cite{BO}
and the fact that, for $(f,g)\in R(E)$,
$\Vert(f,g)\Vert_{H^{2+s}(\Omega)\times H^{1+t}(\Omega)}\le \Vert(f,g)\Vert$,
because of Lemma~\ref{lem_regul}.

In what follows we will prove \eqref{bound2}:
assume that $T(u_k,\phi_k)=\mu(u_k,\phi_k)$, $k=1,\ldots,m$.
Since $a(\cdot,\cdot)$ is an inner product in $\Hcd\times\Hcu$,
we can choose a dual basis for $R(E^*)$ 
denoted by $(u_k^*,\phi_k^*)\in\Hcd\times\Hcu$ satisfying 
$$a((u_k,\phi_k),(u_l^*,\phi_l^*))=\delta_{k,l}.$$  
Now, from \cite[Theorem~7.2]{BO}, we have that
\begin{equation*}
|\mu-\hat{\mu}_h|\leq \frac{1}{m}\sum\limits_{k=1}^{m}|\langle(T-T_h)(u_k,\phi_k),(u_k^*,\phi_k^*)\rangle|
+C||(T-T_h)|_{R(E)}|| ||(T^*-T_h^*)|_{R(E^*)}||,
\end{equation*}
where $\langle\cdot,\cdot\rangle$ denotes the corresponding duality pairing. 

Thus, in order to obtain \eqref{bound2}, we need to bound
the two terms on the right hand side above.

The second term can be easily bounded from Lemmas~\ref{propP1} and \ref{propP1adj}.
In fact, we have 
\begin{equation}\label{est1_conv_autv}
||(T-T_h)|_{R(E)}|| ||(T^*-T_h^*)|_{R(E^*)}||\leq Ch^{2\min\{s,t\}}.
\end{equation}

Next, we manipulate the first term as follows:
adding and subtracting $(v_h,\psi_h)\in Z_h$
and using the definition of $T$ and $T_h$, we obtain,
\begin{align}
&\langle(T-T_h)(u_k,\phi_k),(u_k^*,\phi_k^*)\rangle
=a((T-T_h)(u_k,\phi_k),(u_k^*,\phi_k^*))\nonumber\\
&=a((T-T_h)(u_k,\phi_k),(u_k^*,\phi_k^*)-(v_h,\psi_h))+a(T(u_k,\phi_k),(v_h,\psi_h))
-a(T_h(u_k,\phi_k),(v_h,\psi_h))\nonumber\\
&=a((T-T_h)(u_k,\phi_k),(u_k^*,\phi_k^*)-(v_h,\psi_h))+b((u_k,\phi_k),(v_h,\psi_h))
-a(T_h(u_k,\phi_k),(v_h,\psi_h))\nonumber\\
&+a_h(T_h(u_k,\phi_k),(v_h,\psi_h)) -b_h((u_k,\phi_k),(v_h,\psi_h))\nonumber\\
&=\Big\{a((T-T_h)(u_k,\phi_k),(u_k^*,\phi_k^*)-(v_h,\psi_h))\Big\}
+\Big\{b((u_k,\phi_k),(v_h,\psi_h))-b_h((u_k,\phi_k),(v_h,\psi_h))\Big\}\nonumber\\
&+\Big\{a_h(T_h(u_k,\phi_k),(v_h,\psi_h))- a(T_h(u_k,\phi_k),(v_h,\psi_h))\Big\}
\qquad \forall (v_h,\psi_h)\in Z_h.\label{exp_Rod}
\end{align}  
Now, we estimate each bracket in \eqref{exp_Rod} separately.
First, to bound the second bracket, we use the additional regularity of
$(u_k,\phi_k)\in R(E)\subset H^{2+s}(\O)\times H^{1+t}(\O)$
and repeating the same steps used to derive \eqref{exp-E_1}
(in this case with $(u_k,\phi_k)$ instead of $(f,g)$), we have
\begin{align}
& b_{h,K}((u_k,\phi_k),(v_h,\psi_h))-b_{\E}((u_k,\phi_k),(v_h,\psi_h))
=E_{11}+E_{12}+E_{13}+E_{14}.\nonumber
\end{align}

Now, we will bound each term $E_{1i} \ i=1,2,3,4$,
as in the proof of Lemma~\ref{lemcotste},
but in this case exploiting the additional regularity
and the estimates in Lemmas~\ref{lem_regul} and \ref{lem_regul_adj}
for $(u_k,\phi_k)\in R(E)$ and $(u_k^{*},\phi_k^{*})\in R(E^{*})$,
respectively.

In particular, the terms $E_{11},E_{12}$ and $E_{14}$
can be bound exactly as in the proof of Lemma~\ref{lemcotste}.
However, for the term $E_{13}$, we proceed as follows:

\begin{align*}
E_{13}&=\int_K \nabla \Pi_1^{\nabla}\phi_k \cdot \nabla \overline{v_h} - \int_K \nabla \phi_k
\cdot \nabla \overline{v_h}=\int_K \nabla(\Pi_1^{\nabla}\phi_k -\phi_k )\cdot \nabla \overline{v_h}\\[1ex]
&=\int_K\nabla(\Pi_1^{\nabla}\phi_k -\phi_k )\cdot \nabla (\overline{v_h}
-\tilde{v}_{h}^{\pi})\leq |\Pi_1^{\nabla}\phi_k -\phi_k  |_{1,K}|v_h -\tilde{v}_{h}^{\pi} |_{1,K}\\[1ex]
&=\inf_{q_h\in\mathbb{P}_1(K)}|\phi_k-q_h|_{1,K}|v_h -\tilde{v}_{h}^{\pi} |_{1,K}
\leq Ch_K^{1+t}|\phi_k |_{1+t,K}|v_h |_{2,K}\\
&\leq Ch_K^{2\min\{s,t\}}|\phi_k |_{1+t,K}|v_h |_{2,K},
\end{align*}
where we have used the definition of $\Pi_1^{\nabla}$
with $\tilde{v}_{h}^{\pi}\in\mathbb{P}_1(K)$ such that Proposition~\ref{app1}
holds true and the fact that $\phi_k\in H^{1+t}(\O)$ together with Proposition~\ref{app1} again.

Therefore taking sum and using the additional regularity for $\phi_k$,
together with Lemma~\ref{lem_regul}, we obtain
\begin{align}
&\Big\{b((u_k,\phi_k),(v_h,\psi_h))-b_h((u_k,\phi_k),(v_h,\psi_h))\Big\}
\leq Ch^{2\min\{s,t\}}||(u_k,\phi_k )||||(v_h,\psi_h) ||
\qquad \forall (v_h,\psi_h) \in Z_h.\label{eq_Est_Rod0}
\end{align}

Now, we estimate the third bracket in \eqref{exp_Rod}.
Let $(w_h,\xi_h):=T_h(u_k,\phi_k)$ and
$\Pi_h^K$ be defined by $(\Pi_h^K(v,\psi))|_K:=(\Pi_2^{\Delta}v,\Pi_1^{\nabla}\psi)$
for all $K\in\CT_h$ and for all $(v,\psi)\in\Hcd\times\Hcu$, where $\Pi_2^{\Delta}$ and $\Pi_1^{\nabla}$
have been defined in \eqref{numero}-\eqref{numeroo} and \eqref{numero1}-\eqref{numeroo1}, respectively.
Hence, we have
\begin{align}
&a_h((w_h,\xi_h),(v_h,\psi_h))-a((w_h,\xi_h),(v_h,\psi_h))=\sumkth
\Big\{ a_{h,K}((w_h,\xi_h),(v_h,\psi_h))-a_\E((w_h,\xi_h),(v_h,\psi_h)) \Big\}\nonumber\\
&=\sumkth \Big\{ a_{h,K}((w_h,\xi_h)-(\Pi_2^{\Delta}w_h,\Pi_1^{\nabla}\xi_h),(v_h,\psi_h))
+a_K((\Pi_2^{\Delta}w_h,\Pi_1^{\nabla}\xi_h)-(w_h,\xi_h),(v_h,\psi_h))\Big\}\nonumber\\
&=\sumkth \Big\{ a_{h,K}((w_h,\xi_h)-(\Pi_2^{\Delta}w_h,\Pi_1^{\nabla}\xi_h),(v_h,\psi_h)
-(\Pi_2^{\Delta}v_h,\Pi_1^{\nabla}\psi_h))\nonumber\\
&+a_K((\Pi_2^{\Delta}w_h,\Pi_1^{\nabla}\xi_h)-(w_h,\xi_h),(v_h,\psi_h)-(\Pi_2^{\Delta}v_h,\Pi_1^{\nabla}\psi_h)) \Big\}\nonumber\\
&\leq C\sumkth \Big\{|(w_h,\xi_h)-(\Pi_2^{\Delta}w_h,\Pi_1^{\nabla}\xi_h)|_{H^2(K)\times H^{1}(K)}
|(v_h,\psi_h)-(\Pi_2^{\Delta}v_h,\Pi_1^{\nabla}\psi_h)|_{H^2(K)\times H^{1}(K)}\Big\}\nonumber\\
&= C\sumkth \Big\{|T_h(u_k,\phi_k)-\Pi_h^K T_h(u_k,\phi_k)|_{H^2(K)\times H^{1}(K)}
|(v_h,\psi_h)-\Pi_h^K(v_h,\psi_h)|_{H^2(K)\times H^{1}(K)}\Big\},\label{eq_Est_Rod1}
\end{align}
for all $(v_h,\psi_h)\in Z_h$, where we have used \eqref{consis-a1}-\eqref{consis-a2},
Cauchy-Schwarz inequality and \eqref{stab-a1}-\eqref{stab-a2}. Now, using
the triangular inequality, we have that
\begin{align}
|T_h(u_k,\phi_k)-\Pi_h^K T_h(u_k,\phi_k)|_{H^{2}(K)\times H^{1}(K)}&\leq 
|T_h(u_k,\phi_k)-T(u_k,\phi_k)|_{H^{2}(K)\times H^{1}(K)}\nonumber\\
&+|\Pi_h^K T_h(u_k,\phi_k)-\Pi_h^K T(u_k,\phi_k)|_{H^{2}(K)\times H^{1}(K)}\nonumber\\
&+|\Pi_h^K T(u_k,\phi_k)-T(u_k,\phi_k)|_{H^{2}(K)\times H^{1}(K)}.\nonumber
\end{align}
Thus, from \eqref{eq_Est_Rod1}, the above estimate,
the stability of $\Pi_h^K$ and the additional regularity for $(u_k,\phi_k)$ together
with Lemma~\ref{lem_regul}, we have
\begin{align}
&a_h(T_h(u_k,\phi_k),(v_h,\psi_h))- a(T_h(u_k,\phi_k),(v_h,\psi_h))\nonumber\\
&\leq Ch^{\min\{s,t \}}||(u_k,\phi_k )||\sumkth
|(v_h,\psi_h)-\Pi_h^K(v_h,\psi_h)|_{H^2(K)\times H^{1}(K)}\qquad \forall (v_h,\psi_h)\in Z_h. \label{eq_Est_Rod2}
\end{align}

Finally, we take $(v_h,\psi_h):=\mathcal{P}(u_k^*,\phi_k^*)\in Z_h$ in \eqref{exp_Rod}.
Thus, on the one hand, we bound the first bracket in \eqref{exp_Rod} as follows,
\begin{align}
&a((T-T_h)(u_k,\phi_k),(u_k^*,\phi_k^*)-(v_h,\psi_h))=a((T-T_h)(u_k,\phi_k),(u_k^*,\phi_k^*)
-\mathcal{P}(u_k^*,\phi_k^*))\nonumber\\
&\leq |(T-T_h)(u_k,\phi_k)|_{\Hcd\times\Hcu}|(u_k^*,\phi_k^*)-\mathcal{P}(u_k^*,\phi_k^*)|_{\Hcd\times\Hcu}\nonumber\\
&= |(T-T_h)(u_k,\phi_k)|_{\Hcd\times\Hcu}\inf\limits_{(r_h,s_h)\in Z_h}|(u_k^*,\phi_k^*)
-(r_h,s_h)|_{\Hcd\times\Hcu}\nonumber\\
&\leq  |(T-T_h)(u_k,\phi_k)|_{\Hcd\times\Hcu}|(u_k^*,\phi_k^*)-((u_k^*)_I,(\phi_k^*)_I)|_{\Hcd\times\Hcu}\nonumber\\
&\leq Ch^{2{\min\{s,t \}}}||(u_k^*,\phi_k^*)||,\nonumber
\end{align}
where we have used \eqref{best_appr_of_P}, Propositions~\ref{app0} and \ref{app2},
the additional regularity for $(u_k^*,\phi_k^*)$, Lemma~\ref{lem_regul_adj}
and Lemma~\ref{propP1}.

On the other hand, from \eqref{eq_Est_Rod2} we have that
\begin{align}
|(v_h,\psi_h)-\Pi_h^K(v_h,\psi_h)|&_{H^2(K)\times H^{1}(K)}=|\mathcal{P}(u_k^*,\phi_k^*)-\Pi_h^K\mathcal{P}(u_k^*,\phi_k^*)|_{H^2(K)\times H^{1}(K)}\nonumber\\
&\leq |\mathcal{P}(u_k^*,\phi_k^*)-(u_k^*,\phi_k^*) |_{H^2(K)\times H^{1}(K)}+|(u_k^*,\phi_k^*)-\Pi_h^K(u_k^*,\phi_k^*) |_{H^2(K)\times H^{1}(K)}\nonumber\\
&+|\Pi_h^K((u_k^*,\phi_k^*)-\mathcal{P}(u_k^*,\phi_k^*))|_{H^2(K)\times H^{1}(K)}.\nonumber
\end{align}
Then, using again \eqref{best_appr_of_P}, Propositions~\ref{app0} and \ref{app2},
the additional regularity for $(u_k^*,\phi_k^*)$, Lemma~\ref{lem_regul_adj}
and Lemma~\ref{propP1}, we obtain from \eqref{eq_Est_Rod2} that
\begin{align}
&a_h(T_h(u_k,\phi_k),(v_h,\psi_h))- a(T_h(u_k,\phi_k),(v_h,\psi_h))\leq
Ch^{2{\min\{s,t \}}}||(u_k,\phi_k )||||(u_k^*,\phi_k^*)||. \label{eq_Est_Rod3}
\end{align}
Thus, from \eqref{exp_Rod}, \eqref{eq_Est_Rod0} and \eqref{eq_Est_Rod3},
we obtain  
\begin{align}
|\langle (T-T_h)(u_k,\phi_k),(u_k^*,\phi_k^*) \rangle|\leq C h^{2\min\{s,t \}} \label{preend}.
\end{align}
Therefore, the proof follows from estimates \eqref{est1_conv_autv} and \eqref{preend}.
\end{proof}

\begin{remark}
The error estimate for the eigenvalue $\mu$ of $T$ yield
analogous estimate for the approximation
of the eigenvalue $\lambda = 1/\mu$ of Problem~\ref{P1}
by means of $\hat{\lambda}_h:=\frac{1}{m}\sum\limits_{k=1}^{m}\lambda_h^{(k)}$,
where $\lambda_h^{(k)}=1/\mu_h^{(k)}$.
\end{remark}

\setcounter{equation}{0}
 \section{Numerical results}\label{SEC:NumRes}

In this section we present a series of numerical experiments
to solve the transmission eigenvalue problem with the Virtual
Element scheme \eqref{probspecdiscr}. However, to complete the
choice of the VEM, we had to fix the forms $s_{K}^{\Delta}(\cdot,\cdot)$
and $s_{K}^{\nabla}(\cdot,\cdot)$ satisfying \eqref{term-stab-SK}
and \eqref{term-stab-SK0}, respectively.
For $s_{K}^{\Delta}(\cdot,\cdot)$, we consider the same definition as in 
\cite{MRV}:
\begin{align*}
s_\E^{\Delta}(u_h,v_h):=\sigma_{\E}\sum\limits_{i=1}^{N_\E}[u_h(P_i)v_h(P_i)
+h_{P_i}^2\nabla u_h(P_i)\cdot\nabla v_h(P_i)] & \quad \forall u_h,v_h\in W_h^{K},
\end{align*}
where $P_1,\ldots,P_{N_{\E}}$ are the vertices of $\E$,
$h_{P_i}$ corresponds to the maximum diameter of the elements with $P_i$ as a vertex
and $\sigma_\E>0$ is a multiplicative factor to take into account the magnitude
of the parameter and the $h$-scaling, for instance,
in the numerical tests we have picked $\sigma_\E>0$ as the mean value of
the eigenvalues of the local matrix $a_{\E}^{\Delta}\big(\Pi_2^{\Delta} u_h,\Pi_2^{\Delta} v_h\big)$.
This ensures that the stabilizing term scales as $a_{\E}^{\Delta}(v_h,v_h)$.
Now, a choice for $s_{K}^{\nabla}(\cdot,\cdot)$ is given by
\begin{align*}
s_\E^{\nabla}(\phi_h,\psi_h):=\sum\limits_{i=1}^{N_\E}\phi_h(P_i)\psi_h(P_i)
& \quad \forall \phi_h,\psi_h\in V_h^{K},
\end{align*}
which corresponds to the identity matrix of dimension $N_\E$.
A proof of \eqref{term-stab-SK}
and \eqref{term-stab-SK0} for the above choices could be derived
following the arguments in \cite{BLR2017}.
Finally, we mention that the previous definitions
are in accordance with the analysis presented in \cite{MRR2015,MRV}
in order to avoid spectral pollution.

We have implemented in a MATLAB code the proposed VEM on arbitrary polygonal
meshes, by following the ideas presented in \cite{BBMR2014}.
Moreover, we compare our results with those existing
in the literature, for example
\cite{CakMonkSun2014,ChenGZZ2016,ColtonMonkSun2010,HYBi2017}.
We have considered three different domains, namely:
square domain, a circular domain centered at the origin
and an L-shaped domain.

\subsection{Test 1: Square domain}

In this test, we have taken $\O:=(0,1)^2$
and index of refraction $n=4$ and $n=16$.
We have tested the method by using
different families of meshes
(see Figure~\ref{FIG:VM1}):
\begin{itemize}
\item $\CT_h^1$: triangular meshes;
\item $\CT_h^2$: rectangular meshes;
\item $\CT_h^3$: hexagonal meshes;
\item $\CT_h^4$: non-structured hexagonal meshes made of convex hexagons.
\end{itemize}

The refinement parameter $N$ used to label each mesh is the number of elements
on each edge of the domain.

\begin{figure}[t]
\begin{center}
\begin{minipage}{6.3cm}
\centering\includegraphics[height=6.3cm, width=6.3cm]{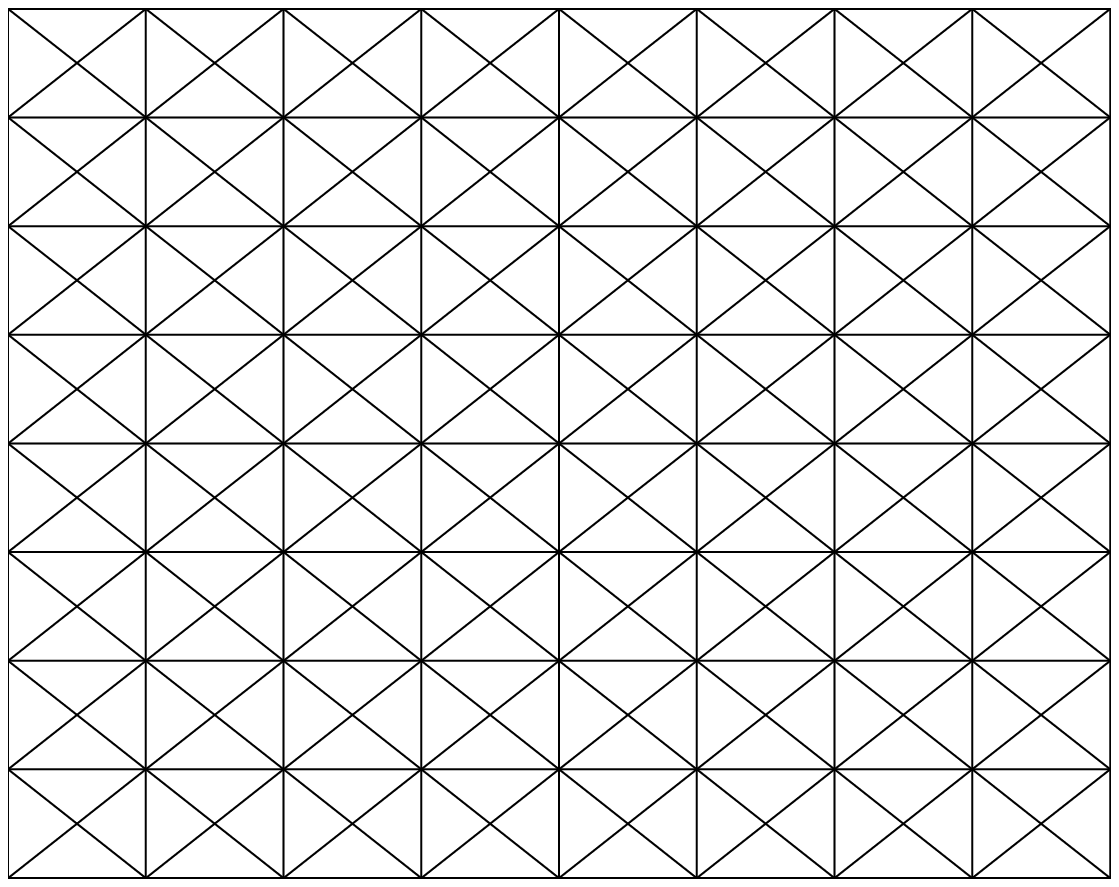}
\end{minipage}
\begin{minipage}{6.3cm}
\centering\includegraphics[height=6.3cm, width=6.3cm]{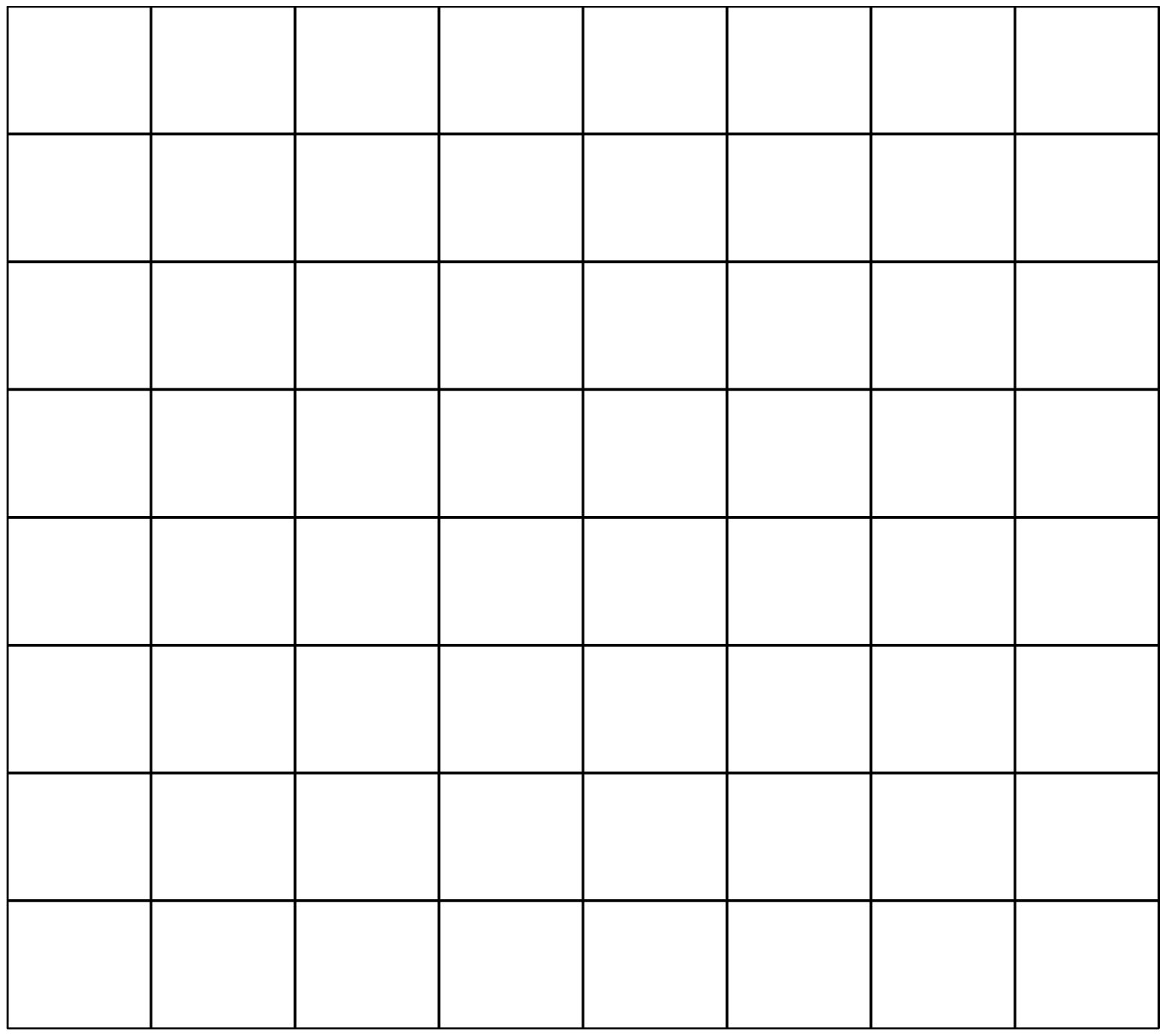}
\end{minipage}
\begin{minipage}{6.3cm}
\centering\includegraphics[height=6.3cm, width=6.3cm]{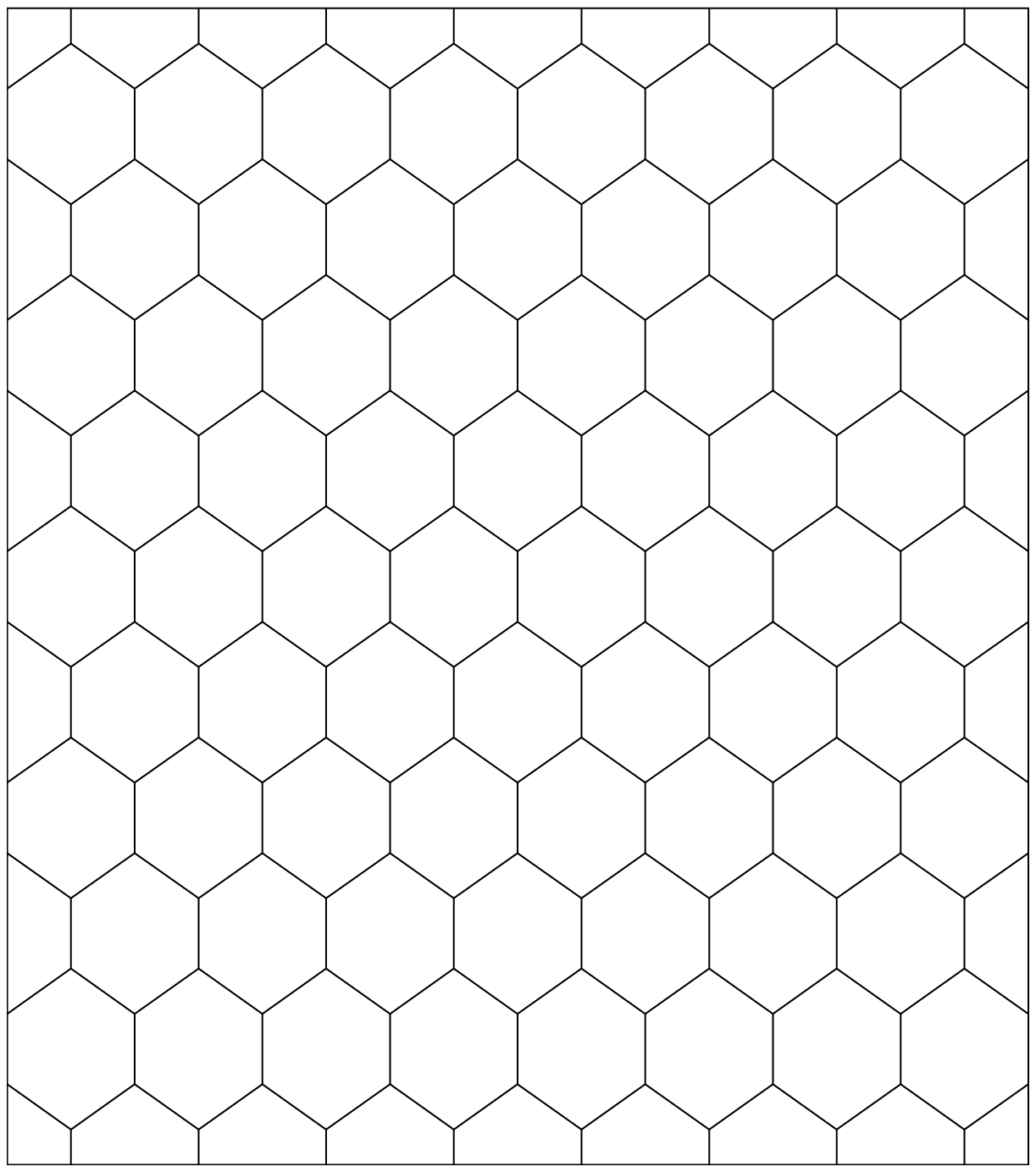}
\end{minipage}
\begin{minipage}{6.3cm}
\centering\includegraphics[height=6.3cm, width=6.3cm]{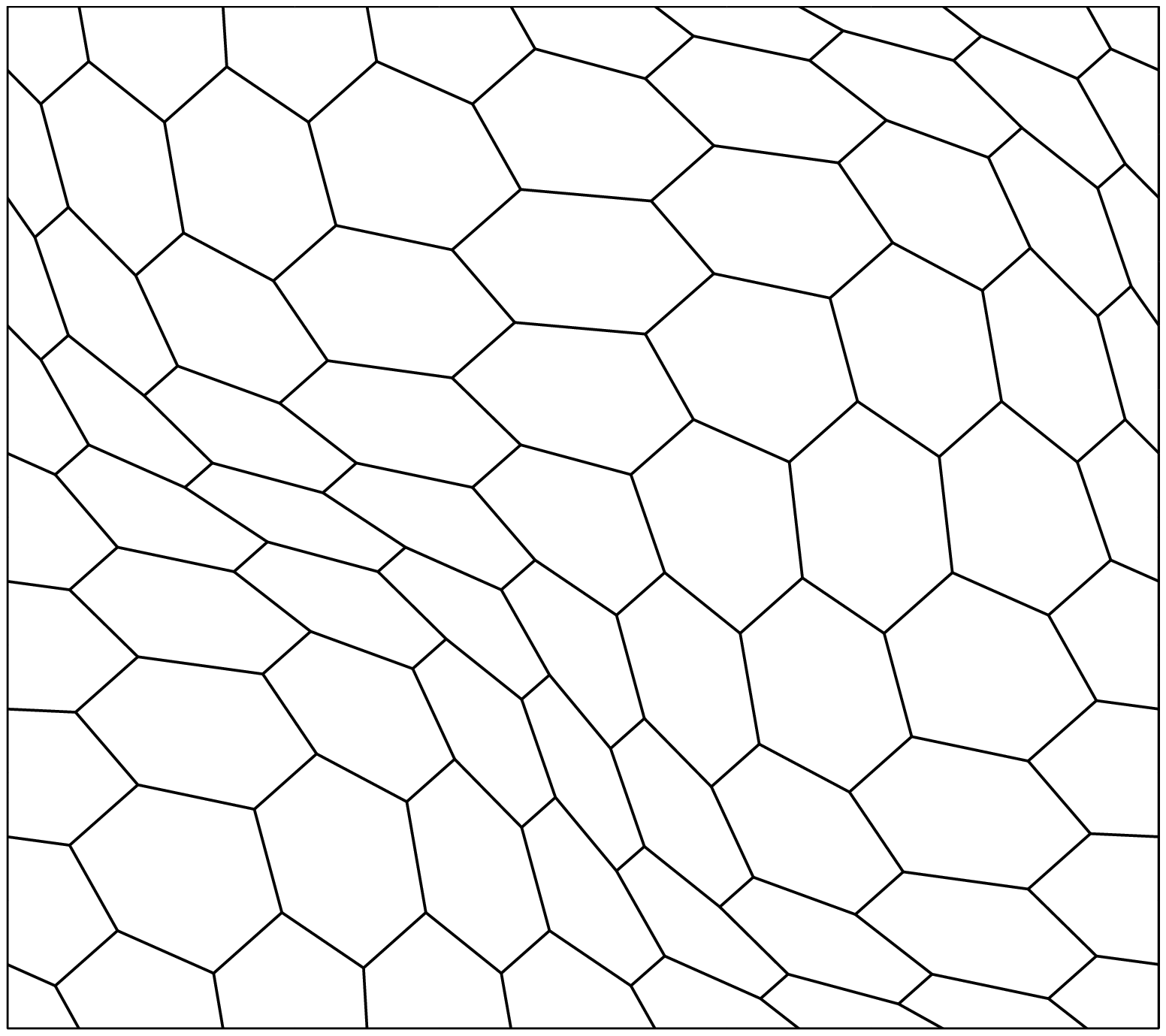}
\end{minipage}
\caption{ Sample meshes: $\CT_h^1$ (top left), $\CT_h^2$ (top right),
$\CT_h^3$ (bottom left) and $\CT_h^{4}$ (bottom right), for $N=8$.}
 \label{FIG:VM1}
\end{center}
\end{figure}

We report in Tables~\ref{TAB:Sq01n=16} and \ref{TAB:Sq01n=4}
the lowest transmission eigenvalues $k_{ih}$, $i=1,2,3,4$
computed by our method with two different families of meshes and $N=32,64,128,$
and with the index of refraction $n=16$ and $n=4$, respectively.
The tables include computed orders of convergence, as well as more accurate
values extrapolated by means of a least-squares fitting. 
Moreover, we compare the performance of the proposed method with those presented
in \cite{ColtonMonkSun2010,HYBi2017}. With this aim, we include
in the last row of Tables~\ref{TAB:Sq01n=16} and \ref{TAB:Sq01n=4}
the results reported in that references, for the same problem.

\begin{table}[ht]
\caption{Test 1: Lowest transmission eigenvalues $k_{ih}$, $i=1,2,3,4$ computed on different meshes
and with index of refraction $n=16$.}
\label{TAB:Sq01n=16}
\begin{center}
{\small\begin{tabular}{|c|lcccc|c|}
\hline
Meshes&$k_{ih}$  	&$k_{1h}$&$k_{2h}$&$  k_{3h}$&$k_{4h}$   \\\hline
&$N=32$   			                          & 1.8805 &  2.4467&    2.4467&    2.8691\\
&$N=64$   		                              & 1.8798 &  2.4449&    2.4449&    2.8671\\
$\CT_h^1$&$N=128$                      & 1.8796 &  2.4444&    2.4444&    2.8666\\
&Order                                       & 2.01   &  2.00  &    2.00  &    2.01 \\
&Extrapolated                                & 1.8796 &  2.4442&    2.4442&    2.8664\\\hline
&$N=32$   			                          & 1.8764 & 2.4318 &   2.4318&    2.8645\\
&$N=64$   		                              & 1.8788 & 2.4410 &   2.4410&    2.8658 \\
$\CT_h^2$&$N=128$                      & 1.8794 & 2.4434 &   2.4434&    2.8663\\
&Order                                       & 1.95   & 1.95   &   1.95  &    1.61 \\
&Extrapolated                                & 1.8796 & 2.4443 &   2.4443&    2.8665
\\\hline\rowcolor{Gray}
&\cite[Argyris method]{ColtonMonkSun2010}    &1.8651  &2.4255  &    2.4271&    2.8178\\\rowcolor{Gray}
&\cite[Continuous method]{ColtonMonkSun2010} &1.9094  &2.5032  &    2.5032&    2.9679\\\rowcolor{Gray}
&\cite[Mixed method]{ColtonMonkSun2010}      &1.8954  &2.4644  &    2.4658&    2.8918\\\rowcolor{Gray}
&\cite{HYBi2017}                             &1.8796  &2.4442  &    2.4442&    2.8664\\\hline 
\end{tabular}}
\end{center}
\end{table}

\begin{table}[ht]
\caption{Test 1: Lowest transmission eigenvalues $k_{ih}$, $i=1,2,3,4$  computed on different meshes
and with index of refraction $n=4$.}
\label{TAB:Sq01n=4}
\begin{center}
{\small\begin{tabular}{|c|lcccc|c|}
\hline 
Meshes&$k_{ih}$  &$k_{1h}$&$k_{2h}$&$  k_{3h}$&$k_{4h}$  \\\hline
&$N=32$ & 4.2835-1.1367&4.2835+1.1367&    5.3373&    5.4172\\
&$N=64$ & 4.2745-1.1446&4.2745+1.1446&    5.4375&    5.4599\\
$\CT_h^3$&$N=128$  & 4.2724-1.1467&4.2724+1.1467&    5.4661&    5.4719\\
& Order  & 2.10\& 1.89  &2.10\& 1.89 &     1.81  &    1.84  \\
&Extrapolated  & 4.2717-1.1475&4.2717+1.1475&    5.4775&    5.4765\\\hline
&$N=32$ 	 & 4.2870-1.1341&4.2870+1.1341&    5.3245&     5.4178\\
&$N=64$   & 4.2753-1.1438&4.2753+1.1438&    5.4329&     5.4602\\
$\CT_h^4$&$N=128$  & 4.2726-1.1465&4.2726+1.1465&    5.4647&     5.4719\\
& Order   & 2.12 \&1.86  &2.12\&1.86   &    1.77  &     1.85   \\
&Extrapolated  & 4.2718-1.1475&4.2718+1.1475&    5.4779&     5.4765
\\\hline\rowcolor{Gray}
&     \cite{HYBi2017}     & 4.2717-1.1474i& 4.2717+1.1474i& 5.4761  & 5.4761\\\hline
\end{tabular}}
\end{center}
\end{table}

It can be seen from Tables~\ref{TAB:Sq01n=16} and \ref{TAB:Sq01n=4}
that the eigenvalue approximation order of
our method is quadratic (as predicted by the theory for convex domains)
and that the results obtained
by the two methods agree perfectly well.

\subsection{Test 2: Circular domain}

In this test, we have taken as domain the circle
$\O:=\{(x,y)\in \R^2: x^2+y^2<1/2\}$.
We have used polygonal meshes created with PolyMesher \cite{malla}
(see Figure~\ref{FIG:CirclePolymesherLap}).
The refinement parameter $N$ is the number of elements
intersecting the boundary.

We report in Table~\ref{TAB:3transp} the five lowest transmission eigenvalues
computed with the virtual element method analyzed in this paper. The
table includes orders of convergence, as well as accurate values
extrapolated by means of a least-squares fitting.
Once again, the last rows show the values obtained
by extrapolating those computed with different methods
presented in \cite{CakMonkSun2014,ChenGZZ2016,ColtonMonkSun2010}.

\begin{table}[ht]
\caption{Test 2: Computed lowest transmission eigenvalues $k_{ih}$, $i=1,2,3,4,5$ with index of refraction $n=16$.}
\label{TAB:3transp}
\begin{center}
{\small\begin{tabular}{|l|ccccc|}
\hline
   & $k_{1h}$ & $k_{2h}$& $k_{3h}$& $k_{4h}$& $k_{5h}$ \\\hline
$N=32$   			                                & 1.9835  & 2.6032  & 2.6037  & 3.2115  & 3.2117\\
$N=64$   		                         & 1.9869  & 2.6105  & 2.6106  & 3.2225  & 3.2227\\
$N=128$                                       & 1.9877  & 2.6123  & 2.6123  & 3.2255  & 3.2256\\
 Order                                              & 1.98    & 1.97    & 2.01    & 1.86    & 1.90 \\
 Extrapolated                                      & 1.9880  & 2.6129  & 2.6129  & 3.2267  & 3.2267\\\hline\rowcolor{Gray}
\cite{CakMonkSun2014}                         & 1.9881  & -       & -       & -       &  -        \\\rowcolor{Gray}
\cite{ChenGZZ2016}                            & 1.9879  & 2.6124  & 2.6124  & 3.2255  & 3.2255  \\\rowcolor{Gray}
\cite[Argyris method]{ColtonMonkSun2010}      & 2.0076  & 2.6382  & 2.6396  & 3.2580  & 3.2598   \\\rowcolor{Gray}
\cite[Continuous method]{ColtonMonkSun2010}    & 2.0301  & 2.6937  & 2.6974  & 3.3744  & 3.3777  \\\rowcolor{Gray}
\cite[Mixed method]{ColtonMonkSun2010}         & 1.9912  & 2.6218  & 2.6234  & 3.2308  & 3.2397   \\\hline
\end{tabular}}
\end{center}
\end{table}

Once more, a quadratic order of convergence can be clearly appreciated from
Table~\ref{TAB:3transp}.

We show in Figure~\ref{FIG:CirclePolymesherLap}
the eigenfunctions corresponding to the four lowest transmission eigenvalues.

\begin{figure}[ht]
\begin{center}
\begin{minipage}{6.3cm}
\centering\includegraphics[height=6cm, width=6.9cm]{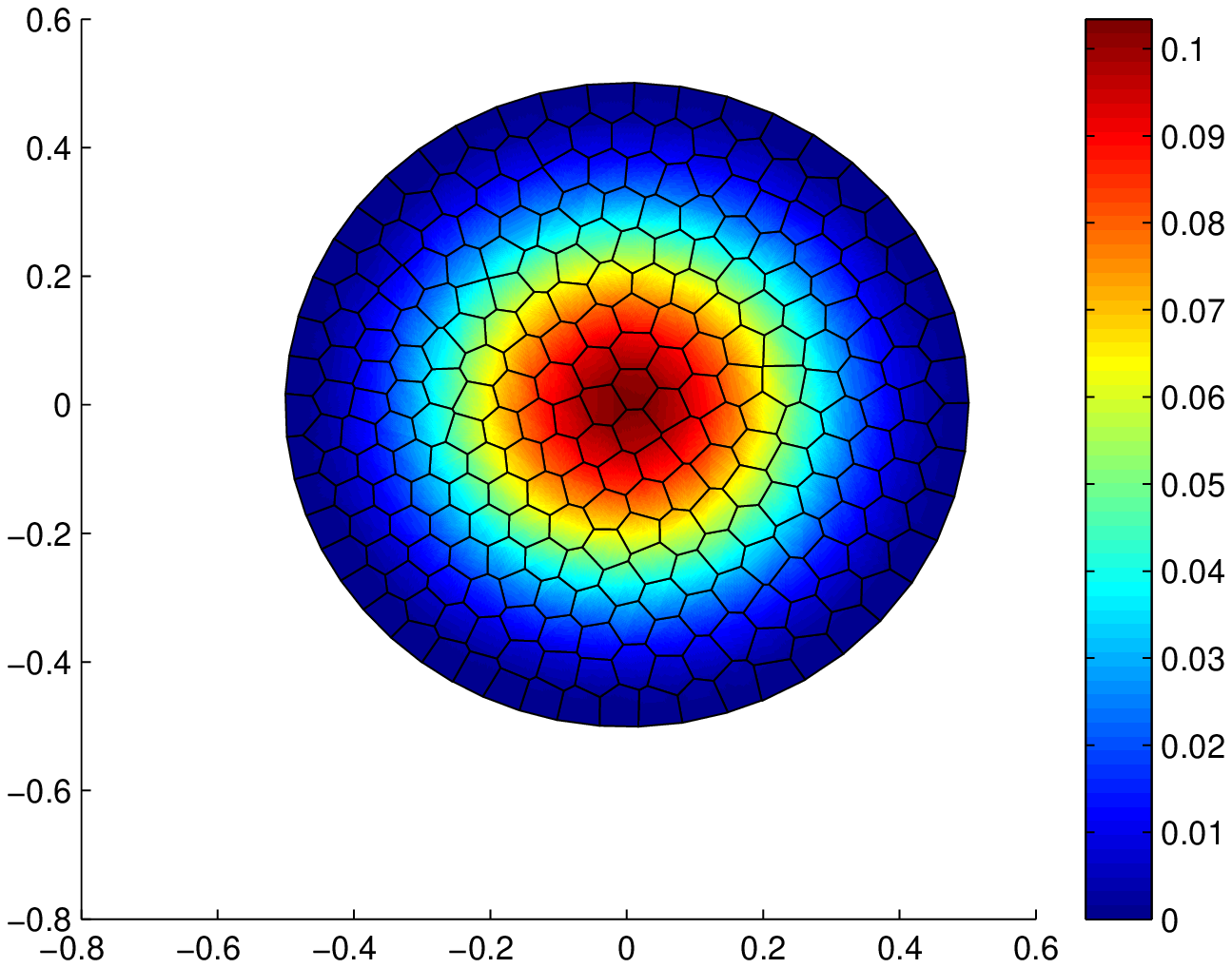}
\end{minipage}
\begin{minipage}{6.3cm}
\centering\includegraphics[height=6cm, width=6.9cm]{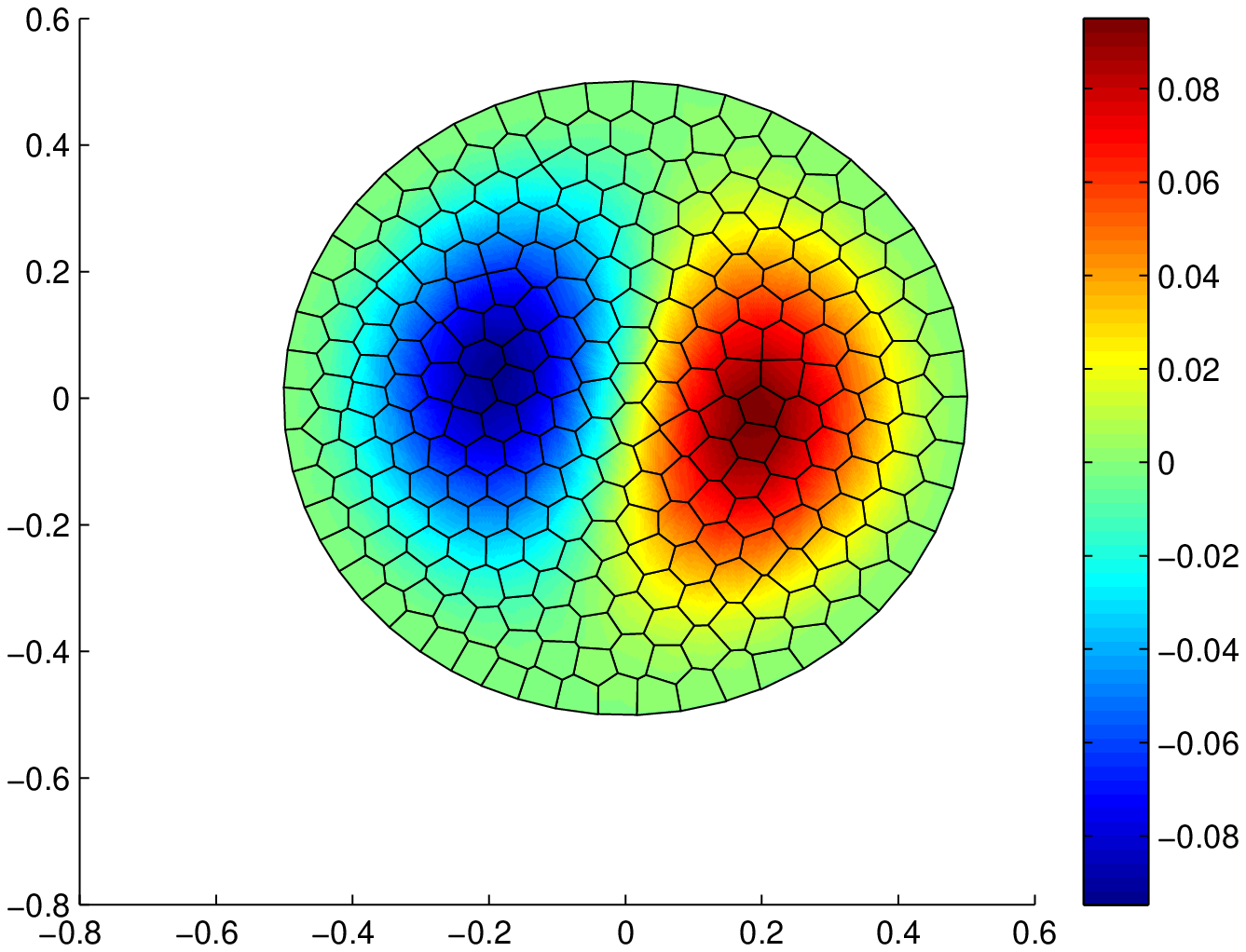}
\end{minipage}
\begin{minipage}{6.3cm}
\centering\includegraphics[height=6cm, width=6.9cm]{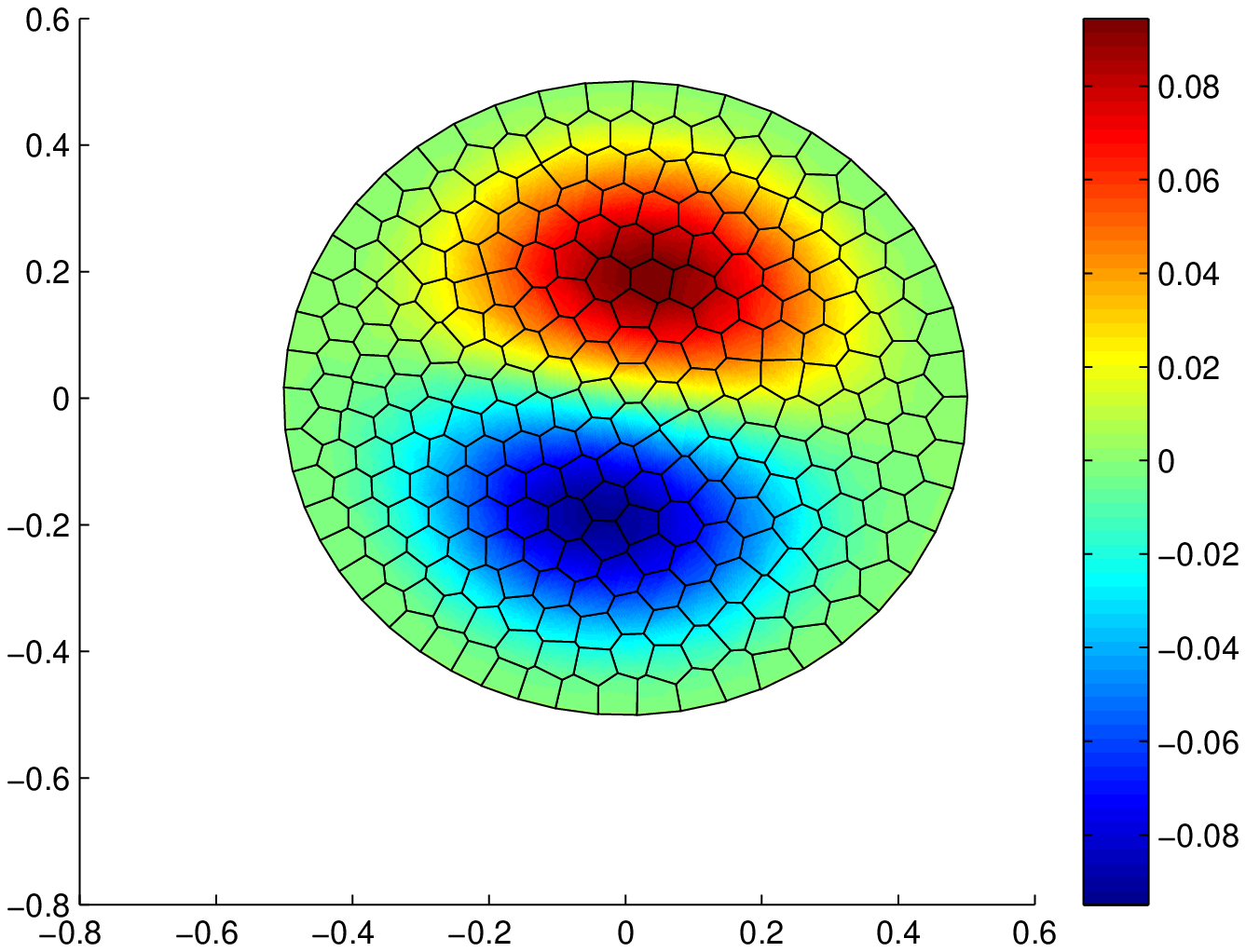}
\end{minipage}
\begin{minipage}{6.3cm}
\centering\includegraphics[height=6cm, width=6.9cm]{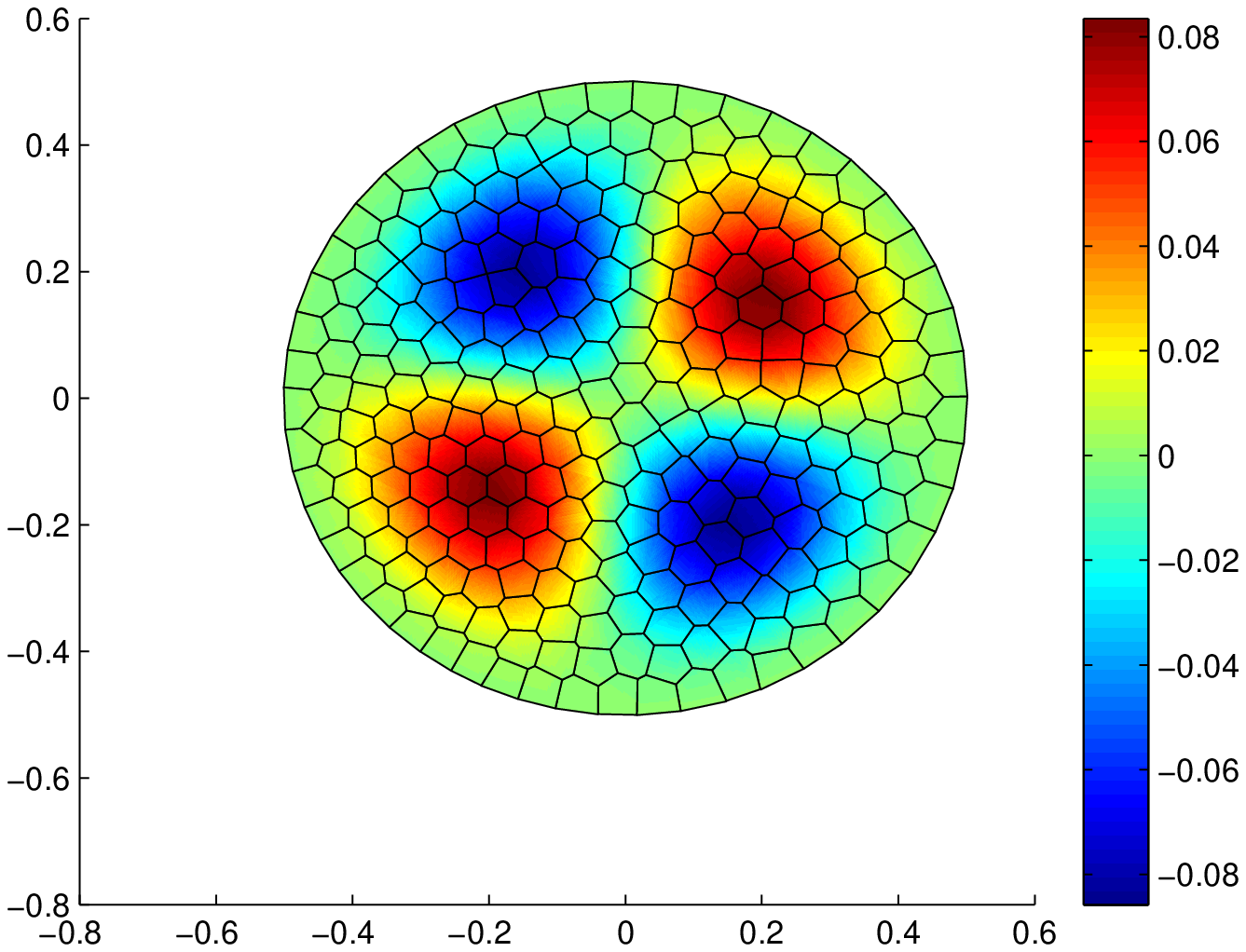}
\end{minipage}
\caption{Test 2. Eigenfunctions $u_{1h}$ (top left),
$u_{2h}$ (top right), $u_{3h}$ (bottom left)
and  $u_{4h}$ (bottom right).}
\label{FIG:CirclePolymesherLap}
\end{center}
\end{figure}	

\subsection{Test 3: L-shaped domain}

Finally, we have considered an L-shaped domain:
$\O:=(-1/2,1/2)^2\backslash([0,1/2]\times[-1/2,0])$.
We have used uniform triangular meshes as those
shown in Figure~\ref{FIG:Square01_n=16_Lshape_unif}.
The meaning of the refinement parameter $N$ is the
number of elements on each edge.

We report in Table~\ref{TAB:LShape05} the four lowest transmission eigenvalues
computed with the virtual scheme analyzed in this paper.
The table includes orders of convergence, as well as accurate
values extrapolated by means of a least-squares fitting.
Once again, we compare the performance of the proposed virtual scheme
with the one presented in \cite{CakMonkSun2014} for the same problem,
using triangular meshes.

\begin{table}[ht]
\caption{Test 3: Computed lowest transmission eigenvalues $k_{ih}$, $i=1,2,3,4$
with index of refraction $n=16$.}
\label{TAB:LShape05}
\begin{center}
{\small\begin{tabular}{|c|lccc|}
\hline
$k_{ih}$   &$k_{1h}$   &$k_{2h}$ &$k_{3h}$&$k_{4h}$ \\\hline
$N=32$     & 2.9690    &3.1480   & 3.4216 &   3.5744\\
$N=64$     & 2.9590    &3.1417   & 3.4136 &   3.5683\\
$N=128$&2.9551 & 3.1400 &   3.4113&    3.5667 \\
 Order      & 1.37     & 1.94    & 1.76  &   2.00  \\
Extrapolated & 2.9527  &  3.1395 & 3.4103 &   3.5662\\\hline\rowcolor{Gray}
\cite{CakMonkSun2014} & 2.9553   & -       &  -     &  -   \\\hline 
\end{tabular}}
\end{center}
\end{table}

We notice that for the first transmission eigenvalue, the method converges with order
close to $\min\{1.089,1.333\}$, which corresponds to the Sobolev regularity
of the domain for the biharmonic equation and Laplace equation
and with homogeneous Dirichlet boundary conditions, respectively (see \cite{G}).
Moreover, the method converges with larger orders for the rest 
of the transmission eigenvalues.

Finally, Figure~\ref{FIG:Square01_n=16_Lshape_unif} shows the eigenfunctions
corresponding to the four lowest transmission eigenvalues with index of refraction  $n=16$.

\begin{figure}[ht]
\begin{center}
\begin{minipage}{6.3cm}
\centering\includegraphics[height=6cm, width=6.9cm]{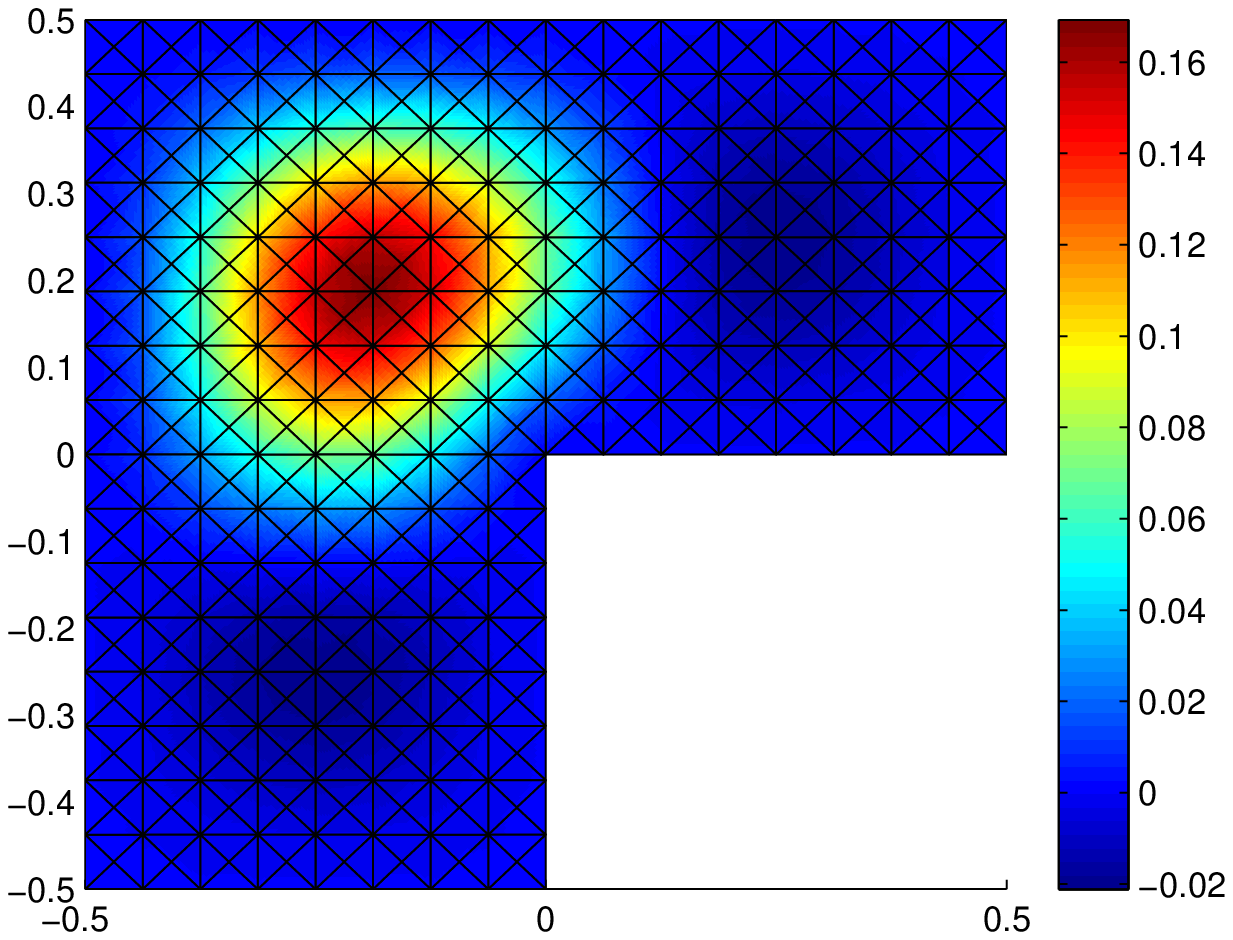}
\end{minipage}
\begin{minipage}{6.3cm}
\centering\includegraphics[height=6cm, width=6.9cm]{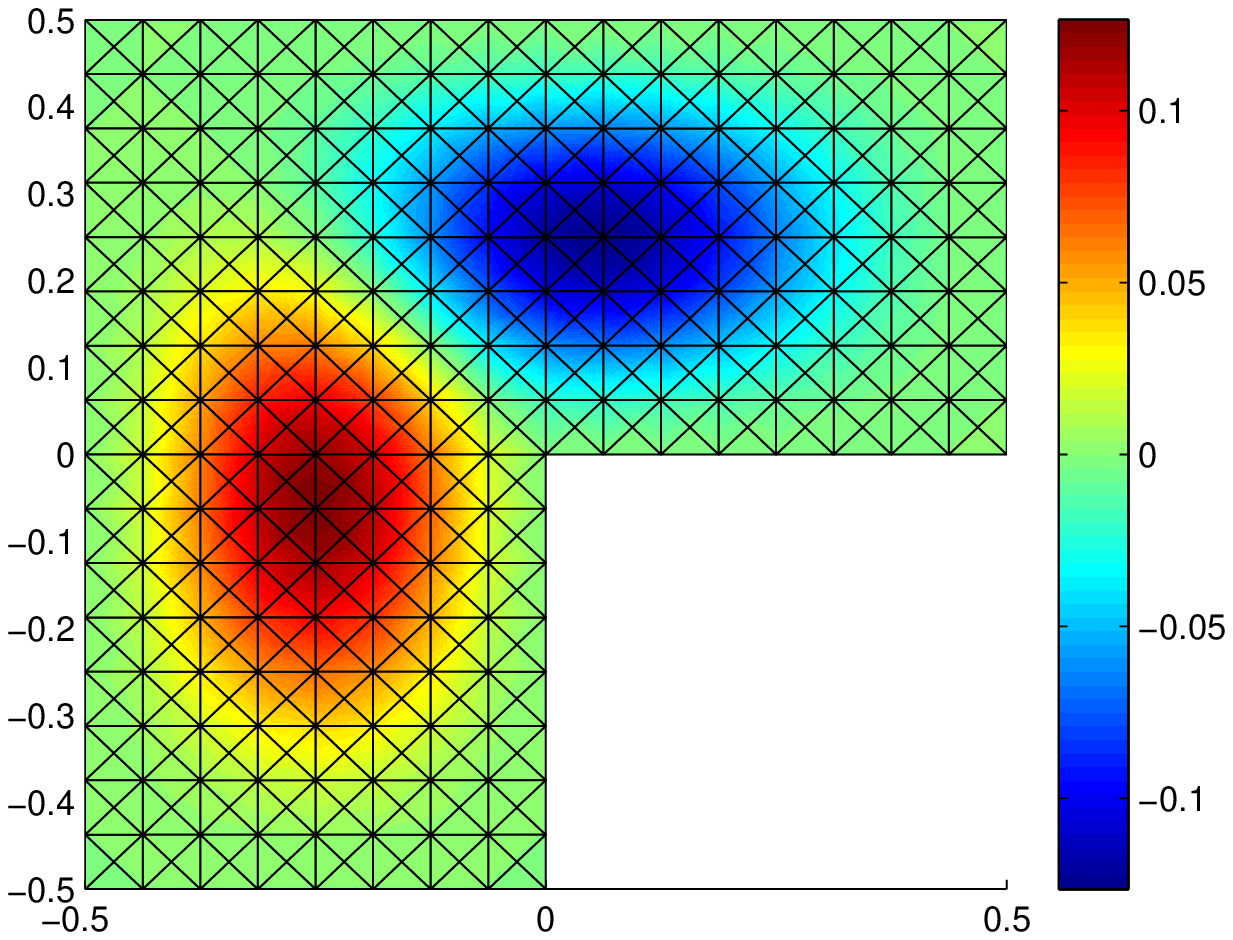}
\end{minipage}
\begin{minipage}{6.3cm}
\centering\includegraphics[height=6cm, width=6.9cm]{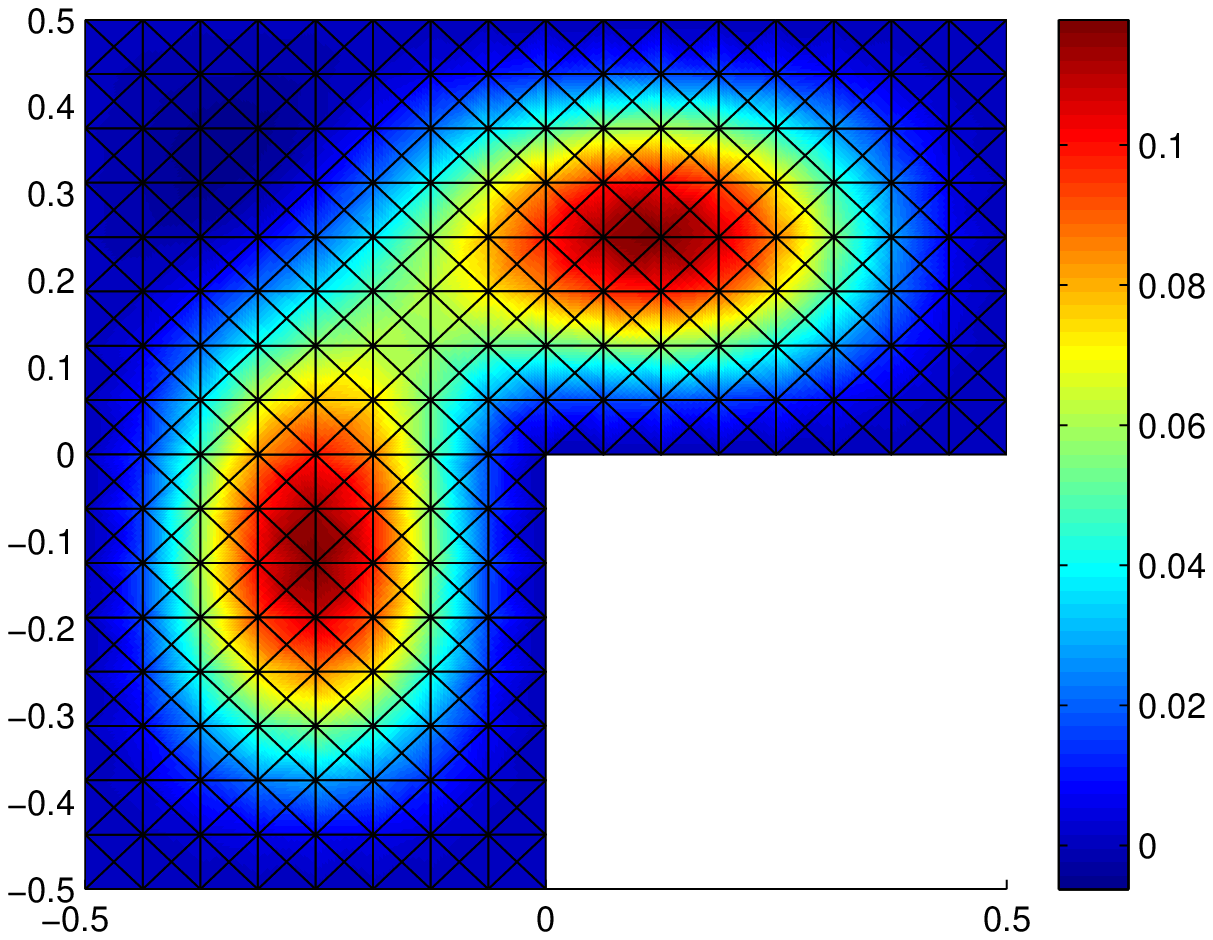}
\end{minipage}
\begin{minipage}{6.3cm}
\centering\includegraphics[height=6cm, width=6.9cm]{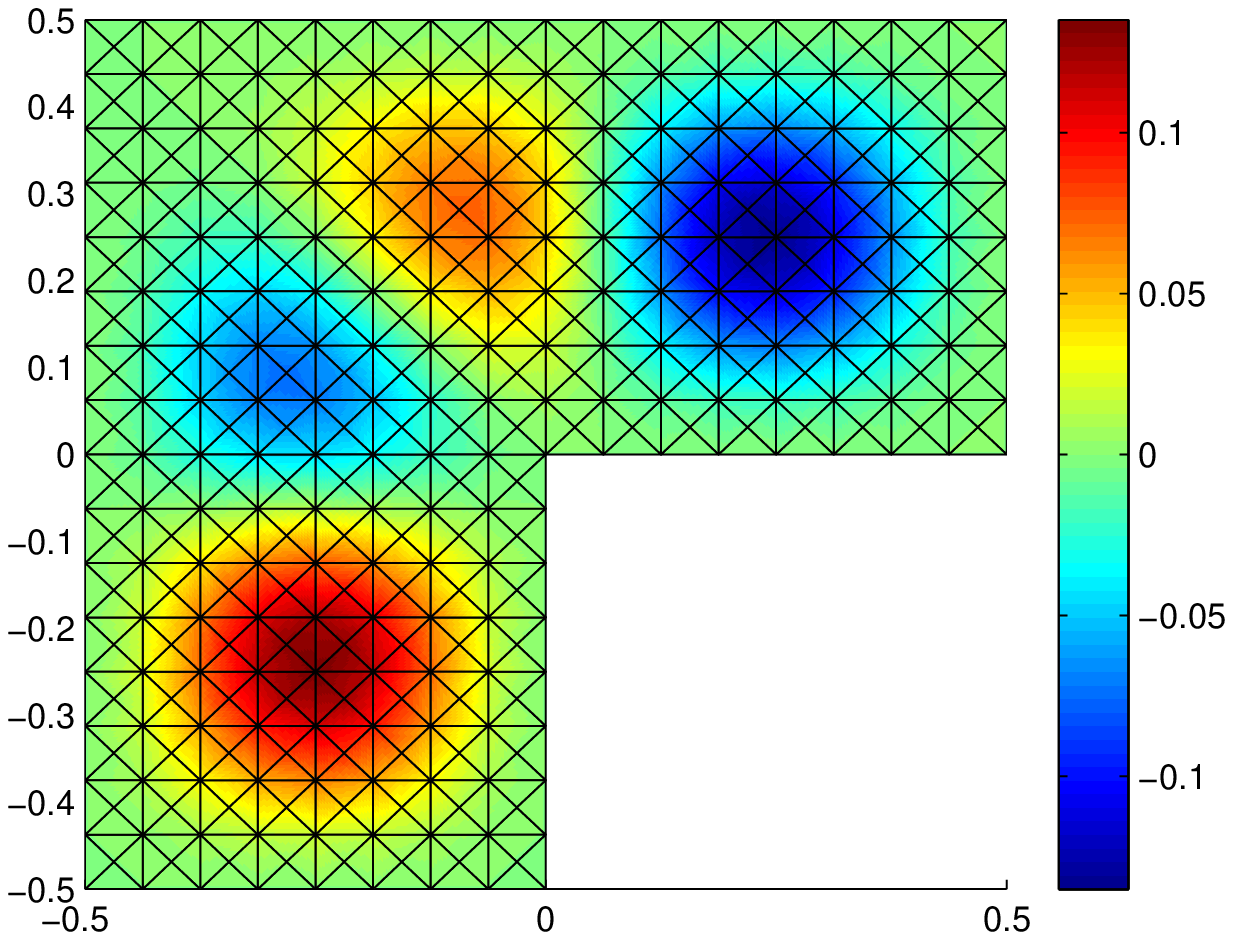}
\end{minipage}
\caption{Test 3. Eigenfunctions $u_{1h}$ (top left), $u_{2h}$ (top right),
$u_{3h}$ (bottom left) and  $u_{4h}$ (bottom right).}
\label{FIG:Square01_n=16_Lshape_unif}
\end{center}
\end{figure}

\section*{Acknowledgments}

The authors are deeply grateful to Prof. Carlo Lovadina
(Universit\`a degli Studi di Milano, Italy) and 
Prof. Rodolfo Rodr\'iguez (Universidad de
Concepci\'on, Chile) for the fruitful discussions.
The First author was partially supported by CONICYT (Chile) through
FONDECYT project 1180913 and by DIUBB through project 171508 GI/VC
Universidad del B\'io-B\'io, Chile.
The second author was partially supported by a CONICYT (Chile)
fellowship.

\bibliographystyle{amsplain}

\end{document}